\documentclass[twoside,11pt]{amsart}
\usepackage[all]{xy}
\CompileMatrices
\usepackage{mathdots}
\usepackage{amsfonts}
\usepackage{amssymb}
\usepackage{amsthm}
\usepackage{amsmath}
\usepackage{ifthen}
\usepackage{enumitem}
\usepackage[usenames]{color}
\usepackage{comment}
\usepackage{appendix}

\usepackage{arydshln}

\usepackage{xcolor}
\usepackage[hyperfootnotes=false, colorlinks, linkcolor={blue}, citecolor={magenta}, filecolor={blue}, urlcolor={blue}, plainpages=false, pdfpagelabels]{hyperref}

 \newtheorem{thm}{Theorem}[section]
 \newtheorem{prop}[thm]{Proposition}
 
 \newtheorem{lem}[thm]{Lemma}

\theoremstyle{definition}
\newtheorem{defn}[thm]{Definition}

\theoremstyle{remark}
\newtheorem{rem}[thm]{Remark}

\newcommand{\Z}{\mathbb{Z}}
\newcommand{\Q}{\mathbb{Q}}
\newcommand{\R}{\mathbb{R}}
\newcommand{\C}{\mathbb{C}}

\newcommand{\GL}{\mathrm{GL}}
\newcommand{\SL}{\mathrm{SL}}

\renewcommand{\b}[1]{\boldsymbol{#1}} % to make Greek letters bold (mathbf doesn't work for them)
\renewcommand{\a}{\mathbf{a}} %archimedean places
  1

\newcommand{\transpose}[1]{\text{$^t\!#1$}}
\newcommand{\G}{\ifmmode {\mathcal{G}}\else${\mathcal{G}}$\ \fi}

\def\sectionnam{\@empty}
\def\subsectionnam{\@empty}

\numberwithin{equation}{section}

\begin{document}

\title[On $p$-adic Measures for Quaternionic Modular Forms] % end with percent
{On $p$-adic Measures for Quaternionic Modular Forms}
%\newline
%\center{\textbf{Preliminary Version}}}
\author{Yubo Jin}
\address{Institute for Advanced Study in Mathematics\\ Zhejiang University\\ Hangzhou, 310058, China}
\email{yubo.jin@zju.edu.cn}
\date{\today}
\maketitle

\begin{abstract}
The purpose of this paper is to study the special values of the standard $L$-functions for quaternionic modular forms using the doubling method. We obtain an integral representation for the $L$-function twisted by a character following the work of Böcherer-Schmidt (\cite{BS}) on Siegel modular forms. As an application, we construct the $p$-adic measure interpolating certain special $L$-values.
\end{abstract}

\tableofcontents

\section{Introduction}

In this paper, we study the special $L$-values of quaternionic modular forms and construct the $p$-adic measure. The main strategy is to study the $L$-functions utilizing an integral representation involving the Siegel Eisenstein series and the properties of special $L$-values then follow from the explicit Fourier expansion of Siegel Eisenstein series. For Siegel modular forms there are two main approaches to obtain the integral representation and construct $p$-adic measures. One is the Rankin-Selberg method (e.g. \cite{CP}), the other is the doubling method (e.g. \cite{BS}). We will mainly follow the doubling method in \cite{BS} to get the integral representation for twisted $L$-functions.

Let $\mathbb{B}=\Q\oplus\Q\zeta\oplus\Q\xi\oplus\Q\zeta\xi$ be a quaternion algebra with $\zeta^2=\alpha,\xi^2=\beta,\zeta\xi=-\xi\zeta$ and $\alpha,\beta\in\Z$. The modular forms studied in this paper are defined on the algebraic group
\[
G=\{g\in\mathrm{SL}_{2n}(\mathbb{B}):g\phi g^{\ast}=\phi\},\phi=\left[\begin{array}{cc}
0 & \epsilon\\
1_n & 0
\end{array}\right],
\]
with $\epsilon=\pm 1$. Let $K$ be a maximal compact subgroup of $G(\R)$. It is well known that $G(\R)/K$ is a Riemannian manifold and is actually Hermitian if $\epsilon=-1,\mathbb{B}\otimes_{\Q}\R\cong\mathbb{H}$ or $\epsilon=1,\mathbb{B}\otimes_{\Q}\R\cong M_2(\R)$. Here $\mathbb{H}$ is the Hamilton quaternion. We will restrict ourselves to above two cases in this paper which we refer as Case I and Case II. More precisely the symmetric space defined in Case I is the Type D domain in \cite{LKW} and the Type C domain in Case II. When $G(\R)/K$ is a Hermitian symmetric space, the algebraicity of modular forms makes sense (see \cite{M}). Indeed, the algebraicity for quaternionic modular forms of Case I is already studied in our previous paper \cite{T}. But in this paper we only consider the case that the associated symmetric space has even degree, i.e. is a `tube' domain.  Instead of defining algebraic modular forms using CM points as in \cite[Section 5.2]{T}, we have an easier description via Fourier expansions. In particular, we have a well-defined action of $\mathrm{Aut}(\C)$ on modular forms in terms of the action on the Fourier coefficients. This will allow us to obtain a more precise algebraic result in Proposition \ref{6.2}. That is we can not only show that the special $L$-values are algebraic up to some constants but also describe the action of $\mathrm{Aut}(\C)$. 

After reviewing the preliminary of quaternionic modular forms and their standard $L$-functions in Section \ref{section 2}, the main computations are carried out in Section \ref{section 3}, \ref{section 4}, \ref{section 5}. Let $f:\mathcal{H}_n\to\C$ be a cusp form on degree $n$ upper half plane. We will define a function $\mathfrak{g}:\mathcal{H}_n\times\mathcal{H}_n\to\C$ deriving form the pullbacks of certain Siegel Eisenstein series on $\mathcal{H}_{2n}$ along $\mathcal{H}_n\times\mathcal{H}_n\to\mathcal{H}_{2n}$. In order to construct the $p$-adic measure we need to define the function $\mathfrak{g}$ such that its Fourier coefficients have natural $p$-adic interpolations, i.e. satisfying the Kummer congruence (Proposition \ref{6.3}). The idea of the construction of $\mathfrak{g}$ is same as the one in \cite{BS}. Indeed, it is same as the twisted Eisenstein series there except for the slight difference for local sections at bad places away from $p$. However, the global computations of the inner product $\langle\mathfrak{g}(z,-\overline{w}),f(z)\rangle$ in \cite{BS} do not trivially generalize to our case since the maximal order $\mathcal{O}$ of $\mathbb{B}$ is not a principal ideal domain in general. For example, the coset decomposition in \cite[Proposition 2.1]{BS} does not work in our case. Therefore we proceed our computations adelically which is also inspired by \cite{LZ,SU}. The adelic formulas \eqref{27},\eqref{32} are reformulated in classical language and the differential operators are applied in Section \ref{section 4}. The differential operator constructed in this paper is similar to the one used in \cite{BS} for Siegel modular forms which is different to the one we used in \cite{T} so the holomorphic projection is avoided. After a final slight modification of $\mathfrak{g}$, we summarize all the analytic results in Section \ref{section 5}. More precisely, with $\mathfrak{g}(z,w)$ defined by \eqref{51},\eqref{52} whose Fourier expansion is computed explicitly in Proposition \ref{5.1}, we obtain an inner product formula of the form 
\[
\begin{aligned}
&\left\langle\left\langle\mathfrak{g}(z,-\overline{w}),f\left|_{\kappa}\left[\begin{array}{cc}
0 & \epsilon\\
1 & 0
\end{array}\right]\right.\right\rangle^z_{\Gamma^0(N^2p)},f\left|_{\kappa}\left[\begin{array}{cc}
1 & 0\\
0 & N^2p
\end{array}\right]\right.\right\rangle^w_{\Gamma^0(N^2p)}\\
=&(\text{certain constants})\cdot\langle f,f\rangle_{\Gamma^0(Np)}L_{Np}(t,f,\chi)E_p(t,\chi)
\end{aligned}
\]
for $t=2n+1,...,\kappa$ where $L_{Np}(s,f,\chi)$ is the partial $L$-function with Euler factors at places dividing $Np$ removed and $E_p(s,\chi)$ the modified $p$-Euler factor defined by \eqref{53}.

As an application, we discuss the properties of special values in the final section. The analytic property (Proposition \ref{6.1}) and algebraic property (Proposition \ref{6.2}) can be obtained immediately from the inner product formula \eqref{54} and the Fourier expansion in proposition \ref{5.1}. For the $p$-adic interpolation, we follow the strategy in \cite{CP} by showing that the Kummer congruences hold for those special $L$-values. 

We fix the following setup. Let $f\in S_{\kappa}(\Gamma_0(N))$ be an eigenform and $0\neq f_0\in S_{\kappa}(\Gamma_0(Np))$ its $p$-stabilisation for a fixed odd prime $p$ with $p\nmid N$. We assume all primes $l\nmid N$ split in $\mathbb{B}$ and $f$ is $p$-ordinary. Denote the $p$-Satake parameter of $f_0$ as $\beta_{1,p},...,\beta_{2n,p}$ and for any integer $L|p^{\infty}$ write $\alpha(L)$ for the eigenvalue of $f_0$ under the action of $U(L)$: $f_0|U(L)=\alpha(L)f_0$. Here $U(L)$ is the Hecke operator defined at the end of Section \ref{section2.4} and $p$-ordinary means that $\alpha(L)\in\mathcal{O}_{\C_p}^{\times}$ is a $p$-adic unit. We state our main result for the $p$-adic interpolation below.

\begin{thm}
(Theorem \ref{6.4},\ref{6.5}) There is a unique $p$-adic measure $\mu$ on $\Z_p^{\times}$ such that:\\
(1) (Case I) For $t=2n+1,...,\kappa$ and primitive Dirichlet character $\chi$ of $p$-power conductor $c_{\chi}=p^{\mathfrak{c}}$ we have
\[
\int_{\Z_p^{\times}}\chi(x)x^{-t+2n}d\mu=\frac{c_{\chi}^{2nt-n(n-1)}}{\alpha(c_{\chi}^2)G(\chi)^n}\frac{\Gamma_{2n}(t)\Lambda_{\infty}(t)}{\pi^{2nt-n(2n-1)}}\frac{E_p(t,\chi)L_{Np}(f,t,\chi)}{\langle f_0,f_0\rangle}.
\]
(2) (Case II) For $t=2n+1,...,\kappa$ and primitive Dirichlet character $\chi$ of $p$-power conductor $c_{\chi}=p^{\mathfrak{c}}$ with $\chi(-1)=(-1)^t$ we have
\[
\begin{aligned}
\int_{\Z_p^{\times}}\chi(x)x^{-t+2n}d\mu&=\frac{c_{\chi}^{(2n-1)t-n(n+1)}}{\alpha(c_{\chi}^2)}\frac{g(\chi)}{G(\chi)^n}\frac{\Gamma(t-2n)\Gamma_{4n}(t)\Lambda_{\infty}(t)}{\pi^{(4n+1)t-n(2n-1)}}\\
&\times\frac{1-\overline{\chi(p)}p^{t-2n-1}}{1-\chi(p)p^{2n-t}}\frac{E_p(t,\chi)L_{Np}(f_0,t,\chi)}{\langle f_0,f_0\rangle},
\end{aligned}
\]
Here:
\begin{itemize}
\item{$L_{Np}(s,f,\chi)$ is the partial $L$-function with Euler factors at places dividing $Np$ removed and $E_p(s,\chi)$ the modified $p$-Euler factor defined by \eqref{53};}
\item{$g(\chi)$ is the usual Gauss sum for the Dirichlet character $\chi$ and $G(\chi)$ is the `quaternionic Gauss sum' of $\chi$ defined by \eqref{57};}
\item{$\Gamma_n(s)=\pi^{n(n-1)}\prod_{i=0}^{n-1}\Gamma(s-2i)$ for Case I, $\Gamma_n(s)=\pi^{\frac{n(n-1)}{4}}\prod_{i=0}^{n-1}\Gamma(s-i/2)$ for Case II and\\
(Case I)
\[
\Lambda_{\infty}(t)=\prod_{i=0}^{\kappa-t-1}C_{2n}\left(t+i-3n+\frac{1}{2}\right)C_{2n}(-t-i)\prod_{i=0}^{n-1}\frac{\Gamma(t+\kappa-2i)}{\Gamma(t+\kappa+n-2i)},
\]
(Case II)
\[
\Lambda_{\infty}(t)=\prod_{i=0}^{\kappa-t-1}C_{2n}\left(t+i-2n+\frac{1}{2}\right)C_{2n}(-t-i)\prod_{i=0}^{2n-1}\frac{\Gamma\left(\frac{t+\kappa-i}{2}\right)}{\Gamma\left(\frac{t+\kappa+2n-i}{2}\right)},
\]
with $C_n(s)=s(s+\frac{1}{2})...(s+\frac{n-1}{2})$.
}
\end{itemize}
\end{thm}

We remark that the $p$-adic aspects of the quaternionic modular forms are not as well understood as those of Siegel modular forms and Hermitian modular forms. For example, one can expect the result in \cite{LZ} can be generalized to modular forms studied here. That is one may hope to construct a $p$-adic $L$-function for `ordinary families of quaternionic modular forms'. Though the analytic computations in \cite{LZ} can be generalized to quaternionic case, the Hida theory as in \cite{Hida} has not established in our case since the Shimura varieties involved in quaternionic case are not of PEL-type. The $p$-adic modular forms as well as the Hida theory for quaternionic Shimura varieties will be explored in author's subsequent research.

\section{Quaternionic Modular Forms and $L$-functions}
\label{section 2}

In this section we review basics of quaternionic modular forms and their (standard) $L$-functions. The references for this kind of modular forms includes \cite{Q,K} for Case I and \cite{G,Sh99} for Case II. For preliminaries on quaternion algebras the reader can refer to \cite{JV}.

\subsection{Quaternion algebra and algebraic groups}

We start by fixing some notations for quaternion algebras and our algebraic groups. 

A quaternion algebra is a central simple algebra of dimension four over $\Q$. Picking up a basis, we can write it in the form
\[
\mathbb{B}=\Q\oplus\Q\zeta\oplus\Q\xi\oplus\Q\zeta\xi,
\]
where
\[
\zeta^2=\alpha,\xi^2=\beta,\zeta\xi=-\xi\zeta,
\]
with $\alpha,\beta$ nonzero squarefree integers. The main involution of $\mathbb{B}$ is given by
\[
\overline{\cdot}:\mathbb{B}\to\mathbb{B}:a+b\zeta+c\xi+d\zeta\xi\mapsto\overline{a+b\zeta+c\xi+d\zeta\xi}=a-b\zeta-c\xi-d\zeta\xi.
\]
We remind the reader that we will also use the same notation $\overline{\cdot}$ for usual complex involution if its meaning is clear from the context. The trace and norm are defined by $\mathrm{tr}(x)=x+\overline{x},N(x)=x\overline{x}$ for $x\in\mathbb{B}$. As usual we denote $M_n(\mathbb{B})$ as a set of $n\times n$ matrices with entries in $\mathbb{B}$. For $X\in M_n(\mathbb{B})$, we write $X^{\ast}=\transpose{\overline{X}},\hat{X}=(X^{\ast})^{-1}$ for the conjugate transpose and its inverse (if makes sense). The determinant $\det$ and trace $\mathrm{tr}$ for matrices in $M_n(\mathbb{B})$ will mean the reduced norm and reduced trace. We use the same notations $\det,\mathrm{tr}$ for usual determinant and trace of matrices whose entries are complex numbers. Denote
\[
\GL_n(\mathbb{B})=\{g\in M_n(\mathbb{B}):\det(g)\neq0\},\quad\mathrm{SL}_n(\mathbb{B})=\{g\in M_n(\mathbb{B}):\det(g)=1\}.
\]
Let $\mathbb{A}$ be the adele ring of $\Q$. For a place $v$ we mean either a finite place corresponding to a prime or the archimedean place $\infty$. The set of finite places is denoted as $\mathbf{h}$. We write $\mathbb{A}=\mathbb{A}_{\mathbf{h}}\R$ with finite adeles $\mathbb{A}_{\mathbf{h}}$ and $x=x_{\mathbf{h}}x_{\infty}$ with $x\in\mathbb{A},x_{\mathbf{h}}\in\mathbb{A}_{\mathbf{h}},x_{\infty}\in\R$. Fix embeddings $\overline{\Q}\to\overline{\Q}_v$ and set $\mathbb{B}_v=\mathbb{B}\otimes_{\Q}\Q_v$. The previous definition of trace, norm, determinant naturally extends locally and adelically. Fix a maximal order $\mathcal{O}$ of $\mathbb{B}$ and set $\mathcal{O}_v=\mathcal{O}\otimes_{\Z}\Z_v$. For a place $v$ we call $v$ splits if $\mathbb{B}_v\cong M_2(\Q_v)$. If this is the case we fix an identification $\mathfrak{i}_v:\mathbb{B}_v\stackrel{\sim}\longrightarrow M_2(\Q_v)$ and assume $\mathfrak{i}_v(\mathcal{O}_v)=M_2(\Z_v)$ for finite places. Write $\mathbb{B}_{\infty}=\mathbb{B}\otimes_{\Q}\R$ for $v=\infty$. We call $\mathbb{B}$ definite if $\mathbb{B}_{\infty}=\mathbb{H}$ is non-split and indefinite if otherwise. Here $\mathbb{H}=\R\oplus\R\mathbf{i}\oplus\R\mathbf{j}\oplus\R\mathbf{ij}$ with $\mathbf{i}^2=\mathbf{j}^2=-1,\mathbf{ij=-ji}$ is the Hamilton quaternion algebra.

We consider following two algebraic groups.\\
(Case I) $\mathbb{B}$ is definite.
\[
G:=G(\Q):=G_n(\Q):=\{g\in\SL_{2n}(\mathbb{B}):gJ_ng^{\ast}=J_n\},\quad J_n=\left[\begin{array}{cc}
0 & -1_n\\
1_n & 0
\end{array}\right].
\]
(Case II) $\mathbb{B}$ is indefinite.
\[
G:=G(\Q):=G_n(\Q):=\{g\in\SL_{2n}(\mathbb{B}):gI_ng^{\ast}=I_n\},\quad I_n=\left[\begin{array}{cc}
0 & 1_n\\
1_n & 0
\end{array}\right].
\]
We also define (connected) similitude groups\\
(Case I)
\[
\tilde{G}:=\{g\in\GL_{2n}(\mathbb{B}):gJ_ng^{\ast}=\nu(g)J_n,\det(g)=\nu(g)^n,\nu(g)\in\Q^+\}.
\]
(Case II)
\[
\tilde{G}:=\{g\in\GL_{2n}(\mathbb{B}):gI_ng^{\ast}=\nu(g)I_n,\det(g)=\nu(g)^n,\nu(g)\in\Q^+\}.
\]

Fix an integral ideal $\mathfrak{n}=(N)$ of $\mathcal{O}$ generated by $N=\prod_vp_v^{n_v}\in\Z$, we define open compact subgroups $K_0(\mathfrak{n})$ and $K^0(\mathfrak{n})$ of $G(\mathbb{A}_{\mathbf{h}})$ by 
\[
K_0(\mathfrak{n}):=\prod_vK_v=\prod_v
\left\{\gamma:=\left[\begin{array}{cc}
a & b\\
c & d
\end{array}\right]\in G(\mathcal{O}_v):\gamma\equiv\left[\begin{array}{cc}
\ast & \ast\\
0_n & \ast
\end{array}\right]\text{ mod }p_v^{n_v}\right\},
\]
\[
K^0(\mathfrak{n}):=\prod_vK_v:=\prod_v
\left\{\gamma=\left[\begin{array}{cc}
a & b\\
c & d
\end{array}\right]\in G(\mathcal{O}_v):\gamma\equiv\left[\begin{array}{cc}
\ast & 0_n\\
\ast & \ast
\end{array}\right]\text{ mod }p_v^{n_v}\right\}.
\]
Denote $\Gamma_0(N):=G(\Q)\cap K_0(\mathfrak{n})$, $\Gamma^0(N):=G(\Q)\cap K^0(\mathfrak{n})$. We have
\begin{equation}
G(\mathbb{A})=G(\Q)G(\R)K_0(\mathfrak{n})=G(\Q)G(\R)K^0(\mathfrak{n}).\label{1} 
\end{equation}
In Case II, this simply follows from the strong approximation. In Case I, the strong approximation does not hold in general. But for subgroups $K_0(\mathfrak{n}),K^0(\mathfrak{n})$, above relation follows by the weak approximation and \cite[Lemma 3.2]{Q}.

\subsection{Symmetric spaces}

Define $S:=S(\Q):=\{X\in M_n(\mathbb{B}):X^{\ast}=-\epsilon X\}$ for the (additive) algebraic group of hermitian matrices (resp. skew-hermitian matrices) with $\epsilon=-1$ (resp. $\epsilon=1$) in Case I (resp. Case II). We use $S^+$ (resp. $S_+$) denote the subgroup of $S$ consisting of positive definite (resp. positive) elements. The symmetric space is defined as
\[
\mathcal{H}:=\mathcal{H}_n:=\{z=x+iy\in  S(\R)\otimes_{\R}\C:x\in S(\R),y\in S^+(\R)\},
\]
with the action of $\tilde{G}(\R)$ given by
\[
g=\left[\begin{array}{cc}
a & b\\
c & d
\end{array}\right]\times z\mapsto (az+b)(cz+d)^{-1}.
\]

Set $\mu(g,z):=cz+d$ and $j(g,z):=\det[\mu(g,z)]$ for the automorphy factor. Put $\delta(z):=\det(y)^{1/2}$ so $\delta(gz)=|j(g,z)|^{-1}\delta(z)$.

For Case I, the symmetric space is same as the quaternionic upper half plane in \cite{Q} and we denote $z_0:=i\cdot 1_n$ for the origin of $\mathcal{H}$. By embedding the quaternion algebra $\mathbb{H}$ into $M_2(\C)$ one can show that this domain is of type D domain in \cite{LKW}, see also \cite[Section 2.2]{T} for more realization of this space.

For Case II, since $\mathbb{B}$ is indefinite we have an isomorphism $\mathbb{B}_{\infty}\cong M_2(\R)$. For example, without lossing generality, write $\mathbb{B}_{\infty}=\R\oplus\R\zeta\oplus\R\xi\oplus\R\zeta\xi$ with $\zeta^2=1,\xi^2=-1,\zeta\xi=-\xi\zeta$. The isomorphism can be given by
\[
\mathfrak{i}:\mathbb{B}_{\infty}\stackrel{\sim}\longrightarrow M_2(\R),\quad x=a+b\zeta+c\xi+d\zeta\xi\mapsto\left[\begin{array}{cc}
a+b & c+d\\
-c+d & a-b
\end{array}\right].
\]
Clearly $\mathfrak{i}(x^{\ast})=-J_1\transpose{\mathfrak{i}(x)}J_1$. Extending this isomorphism entries by entries we obtain 
\[
\mathfrak{i}:M_n(\mathbb{B})\stackrel{\sim}\longrightarrow M_{2n}(\R),\quad X=(x_{ij})\mapsto(\mathfrak{i}(x_{ij})).
\]
We have $\mathfrak{i}(X^{\ast})=-J_n'\transpose{\mathfrak{i}(X)}J_n'$ where $J_n'=\mathrm{diag}[J_1,...,J_1]$ with $n$ copies. This induces the identification
\[
\iota:G(\R)\stackrel{\sim}\longrightarrow\mathrm{Sp}(2n)(\R)=\{g\in\GL_{4n}(\R):gJ_{2n}\transpose{g}=J_{2n}\},
\]
given by $g\mapsto\mathrm{diag}[1,-J_n']\mathfrak{i}(g)\mathrm{diag}[1,J_n']$. It is well known that $\mathrm{Sp}(2n)(\R)$ acts on Siegel upper half plane
\[
\mathfrak{H}_{2n}=\{z=x+iy\in M_{2n}(\R)\otimes_{\R}\C:\transpose{z}=z,y>0\},
\]
by
\[
g=\left[\begin{array}{cc}
a & b\\
c & d
\end{array}\right]\times z\mapsto (az+b)(cz+d)^{-1}.
\]
We fix an identification
\[
\iota:\mathcal{H}_n\stackrel{\sim}\longrightarrow\mathfrak{H}_{2n},z=x+iy\mapsto\mathfrak{i}(x)J_n'+i\mathfrak{i}(y)J_n',
\]
satisfying $\iota(gz)=\iota(g)\iota(z)$ for $g\in G(\R),z\in\mathcal{H}_n$. The origin $z_0$ of $\mathcal{H}_n$ is chosen as $z_0:=i\cdot \xi$ so that $\iota(z_0)=i\cdot 1_{2n}$.

\subsection{Definition of modular forms}

Fix an integer $N$ and $\kappa\in\Z$. The group $\tilde{G}$ acts on function $f:\mathcal{H}\to\C$ by
\[
(f|_{\kappa}\gamma)(z)=\det(\gamma)^{\kappa/2}j(\gamma,z)^{-\kappa}f(\gamma z),\quad \gamma\in\tilde{G},z\in\mathcal{H}.
\]

\begin{defn}
A holomorphic function $f:\mathcal{H}\to\C$ is called a modular form for the congruence subgroup $\Gamma$ of weight $\kappa$ if
\[
f|_{\kappa}\gamma=f\text{ for all }\gamma=\left[\begin{array}{cc}
a_{\gamma} & b_{\gamma}\\
c_{\gamma} & d_{\gamma}
\end{array}\right]\subset\Gamma.
\]
Space of such functions is denoted as $M_{\kappa}(\Gamma)$.
\end{defn}

Denote $e(z):=\mathrm{exp}(2\pi i z)$ for $z\in\C$. Set $\lambda:=\frac{1}{2}\mathrm{tr}$ for Case I and $\lambda:=\mathrm{tr}$ for Case II and $\lambda(x+iy):=\lambda(x)+i\lambda(y)$. For $f\in M_{\kappa}(\Gamma_0(N))$ and $\gamma\in G$, there is a Fourier expansion of the form
\[
(f|_{\kappa}\gamma)(z)=\sum_{\tau\in S_+}c(\tau,\gamma,f)e(\lambda(\tau z)).
\]
We say that $f$ is a cusp form if $c(\tau,\gamma,f)=0$ for any $\gamma$ unless $\tau\in S^+$. The subspace of cusp forms is denoted as $S_{\kappa}(\Gamma_0(N))$. When $\gamma=1$ we omit $\gamma$ in the formula and write
\[
f(z)=\sum_{\tau\in\Lambda}c(\tau)e(\lambda(\tau z)),
\] 
where more precisely, $\Lambda=\{\tau\in S:\lambda(\tau s)\in\Z\text{ for any }s\in S(\mathcal{O})\}$ is a lattice in $S_+$. This condition is equivalent to saying that $\tau_{ii}\in\Z$ and $2\tau_{ij}\in\mathcal{O}^{\#}$. Here $\mathcal{O}^{\#}$ is the dual lattice of $\mathcal{O}$ respect to $\lambda$. 

Let
\[
\mathbf{d}z=\left\{\begin{array}{cc}
\delta(z)^{-2n+1}dz & \text{ Case I, }\\
\delta(z)^{-2n-1}dz & \text{ Case II.}
\end{array}\right.
\] 
be the invariant differential of $\mathcal{H}$. Here $dz=dxdy$ is the usual Euclidean measure. For $f,h\in M_{\kappa}(\Gamma_0(N))$ the Petersson inner product is defined as
\begin{equation}
\langle f,g\rangle:=\langle f,h\rangle_{\Gamma_0(N)}:=\int_{\Gamma_0(N)\backslash\mathcal{H}}f(z)\overline{h(z)}\delta(z)^{2k}\mathbf{d}z,\label{2}
\end{equation}
whenever the integral converges. For example, the integral converges when one of $f,h$ is a cusp form.

We also need the adelic definition of modular forms. Let $\mathfrak{n}=(N)$ be the integral ideal in $\mathcal{O}$ generated by $N$. An (adelic) modular form is defined as a holomorphic function $\mathbf{f}:G(\mathbb{A})\to\C$ satisfying
\[
\mathbf{f}(\alpha gkk_{\infty})=j(k_{\infty},z_0)\mathbf{f}(g),
\] 
for $\alpha\in G,k\in K,k_{\infty}\in K_{\infty}$. Here $K$ is an open compact subgroup of $G(\mathbb{A}_{\mathbf{h}})$ and $K_{\infty}$ is the maximal compact subgroup of $G(\R)$ fixing the origin $z_0$. Denote $\mathcal{M}_{\kappa}(K)$ for the space of such functions and $\mathcal{S}_{\kappa}(K)$ be the subspace of cusp forms. We mainly consider the group $\Gamma_0(N),\Gamma^0(N),K_0(\mathfrak{n}),K^0(\mathfrak{n})$ in this paper. By \eqref{1}, two definitions of modular forms are related by
\[
\mathcal{M}_{\kappa}(K_0(\mathfrak{n}))\cong M_{\kappa}(\Gamma_0(N)),\quad\mathcal{S}_{\kappa}(K_0(\mathfrak{n}))\cong S_{\kappa}(\Gamma_0(N)),\quad\mathbf{f}\leftrightarrow f
\]
with $\mathbf{f}(g)=j(g_{\infty},z_0)^{-k}f(g_{\infty}z_0)$. Same relations hold for changing $\Gamma_0(N)$ to $\Gamma^0(N)$ and $K_0(\mathfrak{n})$ to $K^0(\mathfrak{n})$. For $\mathbf{f,h}\in\mathcal{M}_{\kappa}(K_0(\mathfrak{n}))$ we can also define the Petersson inner product as an integral (if converges)
\begin{equation}
\langle\mathbf{f},\mathbf{h}\rangle:=\langle\mathbf{f},\mathbf{h}\rangle:=\int_{G(\Q)\backslash G(\mathbb{A})}\mathbf{f}(g)\overline{\mathbf{h}(g)}\mathbf{d}g.\label{3}
\end{equation}
Here $\mathbf{d}g$ is the Haar measure of $G(\mathbb{A})$ such that the volume of $K_0(\mathfrak{n})$ is $1$ and $dg_{\infty}=\mathbf{d}(g_{\infty}z_0)$. If $f,h$ are the classical modular form correspond to $\mathbf{f,h}$ then \eqref{2},\eqref{3} are related by
\begin{equation}
\langle\mathbf{f,h}\rangle=\langle f,h\rangle.
\end{equation}
The similar relation also holds for $K^0(\mathfrak{n})$.

\subsection{Hecke operators and $L$-functions}
\label{section2.4}

Fix $N\in\Z$ as above and assume that all primes $l\nmid N$ splits in $\mathbb{B}$. Let 
\[
\mathfrak{M}=\GL_n(\mathbb{B}_{\mathbf{h}})\cap M_n(\mathcal{O}_{\mathbf{h}}),\quad\mathfrak{Q}=\{\mathrm{diag}[\hat{w},w],w\in M\},\quad\mathfrak{X}=K_0(\mathfrak{n})\mathfrak{Q}K_0(\mathfrak{n}).
\]
The Hecke algebra $\mathbb{T}=\mathbb{T}(K_0(\mathfrak{n}),\mathfrak{X})$ is defined as the $\Q$-algebra generated by double cosets $[K_0(\mathfrak{n})\xi K_0(\mathfrak{n})]$ with $\xi\in\mathfrak{X}$. Given $\mathbf{f}\in\mathcal{M}_{\kappa}(K_0(\mathfrak{n}))$ the Hecke operator $[K_0(\mathfrak{n})\xi K_0(\mathfrak{n})]$ acts on $\mathbf{f}$ by
\begin{equation}
(\mathbf{f}|[K_0(\mathfrak{n})\xi K_0(\mathfrak{n})])(g)=\sum_{y\in Y}\mathbf{f}(gy^{-1}),
\end{equation}
where $Y$ is a finite subset of $G_{\mathbf{h}}$ such that
\[
K_0(\mathfrak{n})\xi K_0(\mathfrak{n})=\bigcup_{y\in Y}K_0(\mathfrak{n})y,\quad y=\left[\begin{array}{cc}
a_y & b_y\\
c_y & d_y
\end{array}\right].
\]

The following lemma can be proved as in \cite[Lemma 19.2]{Sh00} and \cite[Lemma 6.1]{Q}.

\begin{lem}
\label{lem2.2}
Let $\xi=\mathrm{diag}[\hat{w},w]\in\mathfrak{Q}_v$. For $v|\mathfrak{n}$, we have
\[
K_v\xi K_v=\bigcup_{d,b}K_v\left[\begin{array}{cc}
\hat{d} & \hat{d}b\\
0 & d
\end{array}\right],
\]
with $d\in \GL_n(\mathcal{O}_v)\backslash\GL_n(\mathcal{O}_v)w\GL_n(\mathcal{O}_v)$, $b\in S(\mathcal{O}_v)/d^{\ast}S(\mathcal{O}_v)d$.
\end{lem}

Let $\mathbf{f}\in\mathcal{S}_{\kappa}(K_0(\mathfrak{n}))$ be an eigenform such that $\mathbf{f}|[K_0(\mathfrak{n})\xi K_0(\mathfrak{n})]=\lambda_{\mathbf{f}}(\xi)\mathbf{f}$ for all $\xi\in\mathfrak{X}$. Let $\chi:\Q^{\times}\backslash\mathbb{A}^{\times}\to\C^{\times}$ be a Hecke character with $\chi^{\ast}:\Z\to\C$ its associate Dirichlet character. The (standard) $L$-function is defined as
\begin{equation}
L(s,\mathbf{f},\chi)=\Lambda^{2n}_N(s,\chi)\sum_{\substack{\xi\in K_0(\mathfrak{n})\backslash\mathfrak{X}/K_0(\mathfrak{n})\\\xi=\mathrm{diag}[\hat{w},w]}}\lambda_{\mathbf{f}}(\xi)\chi^{\ast}(\det(w))\det(w)^{-s}
\end{equation}
with\\
(Case I)
\begin{equation}
\Lambda^{2n}_N(s,\chi)=\prod_{i=0}^{2n-1}L_{N}(2s-2i,\chi^2),
\end{equation}
(Case II)
\begin{equation}
\Lambda^{2n}_N(s,\chi)=L_N(s,\chi)\prod_{i=1}^{2n}L_N(2s-2i,\chi^2).
\end{equation}
Here the subscript $N$ in $L$-function means the Euler factors at $l|N$ are removed. For $\chi=1$, we simply denote $L(s,\mathbf{f}):=L(s,\mathbf{f},1),\Lambda_N^{2n}(s):=\Lambda_N^{2n}(s,1)$. 

\begin{prop}
\label{prop 2.3}
The $L$-function has an Euler product expression
\[
L(s,\mathbf{f},\chi)=\prod_{l}L^{(l)}(s,\mathbf{f},\chi),
\]
with $L^{(l)}(s,\mathbf{f},\chi)$ given by\\
(Case I)
\[
\begin{aligned}
&\prod_{i=1}^{2n}\left((1-\alpha_{i,l}\chi^{\ast}(l)l^{2n-1-s})(1-\alpha_{i,l}^{-1}\chi^{\ast}(l)l^{2n-1-s})\right)^{-1}&l\nmid N,\\
&\prod_{i=1}^{2n}(1-\beta_{i,l}\chi^{\ast}(l)l^{2n-2-s})^{-1}&l|N,l\text{ split },\\
&\prod_{i=1}^{n}(1-\delta_{i,l}\chi^{\ast}(l)l^{2n-1-s})^{-1}&l|N,l\text{ nonsplit },
\end{aligned}
\]
(Case II)
\[
\begin{aligned}
(1-l^{2n-s}\chi^{\ast}(l))^{-1}&\prod_{i=1}^{2n}\left((1-\alpha_{i,l}\chi^{\ast}(l)l^{2n-s})(1-\alpha_{i,l}^{-1}\chi^{\ast}(l)l^{2n-s})\right)^{-1}&l\nmid N,\\
&\prod_{i=1}^{2n}(1-\beta_{i,l}\chi^{\ast}(l)l^{2n-s})^{-1}&l|N,l\text{ split },\\
&\prod_{i=1}^n(1-\delta_{i,l}\chi^{\ast}(l)l^{2n-3-s})^{-1}&l|N,l\text{ nonsplit },
\end{aligned}
\]
for those Satake parameters $\alpha^{\pm}_{1,l},...,\alpha^{\pm}_{2n,l};\beta_{1,l},...,\beta_{2n,l};\delta_{1,l},...,\delta_{n,l}$ of $\mathbf{f}$. 
\end{prop}

\begin{proof}
If $p$ splits in $\mathbb{B}$ we identify the group $G_n(\Q_v)$ with orthogonal group $\mathrm{SO}(n,n)(\Q_v)$ in Case I and symplectic group $\mathrm{Sp}(2n)(\Q_v)$ in Case II. We also identify the local Hecke algebra $\mathbb{T}_v$ with the local Hecke algebra of orthogonal and symplectic group which can be defined in a same way as above. The Satake map is chosen as in \cite[Section 16]{Sh97} for $l\nmid N$, \cite[Proposition 6.3]{Q} for $l$ nonsplit and $l$ split in Case I, \cite[Theorem 2.9]{Sh94} for $l$ in Case II. Then the computation is same as \cite[Theorem 19.8]{Sh00}. For Case I, this is also calculated in \cite[Proposition 6.3]{Q}. However, there is a typo in that computation. Let $l$ split and keep the notation and $\omega_0$ there, we need to calculate
\[
\sum_{d\in\GL_{2n}(\Z_l)\backslash \GL_{2n}(\Q_l)\cap M_{2n}(\Z_l)}\omega_0(\GL_{2n}(\Z_l)d)[S'(\Z_l):d^{\ast}S'(\Z_l)d]|\det(d)|_l^s.
\]
Here 
\[
S'=\{h\in M_{2n}(\Q_p):h^{\ast}=-h\},
\]
so the index $[S'(\Z_l):d^{\ast}S'(\Z_l)d]=|\det(d)|^{-2n+1}_l$ as calculated in \cite[lemma 13.2]{Sh97}. The sum then equals $\prod_{i=1}^{2n}(1-t_il^{2n-2-s})^{-1}$ by of \cite[lemma 16.3]{Sh97}. Similarly, for $p$ nonsplit the term $|\det(d)|_v^{-(2n-1)}$ in that proof should be $|\det(d)|_v^{-(2n+1)}$.
\end{proof}

By strong approximation of $\mathrm{SL}_n(\mathbb{B})$, the representatives of $K_0(\mathfrak{n})\backslash\mathfrak{X}/K_0(\mathfrak{n})$ can be taken in $G(\Q)\cap\mathfrak{X}$. For such $\xi\in G(\Q)\cap\mathfrak{X}$ we have $K_0(\mathfrak{n})\xi K_0(\mathfrak{n})\cap G(\Q)=\Gamma_0(N)\xi \Gamma_0(N)$. Its action on $f\in M_{\kappa}(\Gamma_0(N))$ is given by
\begin{equation}
f|[\Gamma_0(N)\xi\Gamma_0(N)]=\sum_{y\in Y}f|_{\kappa}y,
\end{equation}
if
\[
\Gamma_0(N)\xi\Gamma_0(N)=\bigcup_{y\in Y}\Gamma_0(N)y,\quad y=\left[\begin{array}{cc}
a_y & b_y\\
c_y & d_y
\end{array}\right].
\]
Let $f\in S_{\kappa}(\Gamma_0(N))$ be an eigenform and $\chi^{\ast}$ a Dirichlet character corresponds to the Hecke character $\chi$ one similarly define the $L$-function $L(s,f,\chi^{\ast})$ such that $L(s,f,\chi^{\ast})=L(s,\mathbf{f},\chi)$.

We can also extend the definition of Hecke operators to similitude groups. For example, for an integer $M|N$, we consider another Hecke operator $U(M)$ given by double coset
\[
\Gamma_0(N)\left[\begin{array}{cc}
1 & 0\\
0 & M
\end{array}\right]\Gamma_0(N).
\]
It can be decomposed as
\[
\Gamma_0(N)\left[\begin{array}{cc}
1 & 0\\
0 & M
\end{array}\right]\Gamma_0(N)=\bigcup_{\mathfrak{T}\in S(\mathcal{O})/MS(\mathcal{O})}\Gamma_0(N)\left[\begin{array}{cc}
1 & \mathfrak{T}\\
0 & M
\end{array}\right].
\]
The action of $U(M)$ on $M_{\kappa}(\Gamma_0(N))$ is defined as
\begin{equation}
(f|U(M))(z)=M^{-\frac{n(n+1)}{2}}\sum_{\mathfrak{T}\in S(\mathcal{O})/MS(\mathcal{O})}f(M^{-1}(z+\mathfrak{T})).
\end{equation}
Write the Fourier expansion of $f$ as
\[
f(z)=\sum_{\tau}c(\tau,y)e(\lambda(\tau x)),
\]
with $z=x+iy$. Then
\begin{equation}
(f|U(M))(z)=\sum_{\tau}c(M\tau,M^{-1}y)e(\lambda(\tau x)).
\end{equation}

\subsection{$p$-stabilisation}

Let $f\in S_{\kappa}(\Gamma_0(N))$ be an eigenform. Fix a prime $p\nmid N$, we are going to construct an eigenform $f_0\in S_{\kappa}(\Gamma_0(Np))$ such that $f_0|U(p)=\alpha_0(p)f$ for some $\alpha_0$ and $f_0$ has same eigenvalues with $f$ for Hecke operators away from $p$. The following construction is taken from \cite{CP} for Siegel modular forms. Let $v$ be the place corresponding to $p$. We can extend the definition of Hecke algebra to similitude groups. For simplicity we write $\tilde{G}_v$ for such local group after identified with orthogonal or symplectic group and write $K_v=\tilde{G}_v\cap\GL_{2n}(\Z_v)$. Let $\mathbb{H}_p$ be the local Hecke algebra generated by double coset $[K_v\xi K_v]$ for $\xi\in \tilde{G}_v$. The action of $\mathbb{H}_p$ on $\mathcal{M}_{\kappa}(K_0(\mathfrak{n}))$ is defined similarly as before. There is a well known Satake isomorphism
\[
\omega:\mathbb{H}_p\stackrel{\sim}\longrightarrow\Q[t^{\pm}_0,...,t_{2n}^{\pm}]^{W_{2n}}.
\]
Here $W_{2n}$ is the Weyl group generated by permutation group $S_{2n}$ and mappings:
\[
t_0\mapsto t_0t_j,t_j\mapsto t_j^{-1},t_i\mapsto t_i, (1\leq i\leq 2n,i\neq j).
\]
For example, this can be constructed as in \cite[2.1.7]{CP} so that the restriction to $\mathbb{T}_p$ is same as the one in the proof of Proposition \ref{prop 2.3}. Consider the polynomial $\tilde{Q}(X)\in\Q[t_0,...,t_{2n}][X]$:
\[
\tilde{Q}(X)=(1-t_0z)\prod_{r=1}^{2n}\prod_{1\leq i_1\leq ...\leq i_r\leq 2n}(1-t_0t_{i_1}...t_{i_r}X).
\]
Then there exists a polynomial $Q(X)=\sum_{i=0}^{2^{2n}}(-1)^iT_iX^i\text{ with }T_i\in\mathbb{H}_p$ such that $\omega(Q(X))=\tilde{Q}(X)$. This can be decomposed as
\[
Q(X)=\left(\sum_{i=1}^{2^{2n}-1}V_i(p)X^i\right)(1-U(p)X),V_i(p)=\sum_{j=0}^i(-1)^jT_jU(p)^{i-j}.
\]
Take $\mathbf{f}\in\mathcal{S}_{\kappa}(K_0(\mathfrak{n}))$ be an eigenform corresponding to $f\in S_{\kappa}(\Gamma_0(N))$ with Satake parameters $\alpha_{0,p},...,\alpha_{2n,p}$. Put
\[
\mathbf{f}_0'=\sum_{i=1}^{2^{2n}-1}\alpha_{0,p}^{-i}\mathbf{f}|V_i(p),
\]
then $\mathbf{f}_0'\in \mathcal{S}_{\kappa}(K_0(\mathfrak{n}p^{2^{2n}-1}))$ is an eigenform satisfying $\mathbf{f}_0'|U(p)=\alpha_{0,p}\mathbf{f}_0'$ and has same eigenvalues as $\mathbf{f}$ for Hecke operators away from $p$. This determines an eigenform $f_0'\in S(\Gamma_0(Np^{2^{2n}-1}))$ with same properties. Define
\[
f_0=\sum_{\gamma\in \Gamma_0(Np^{2^{2n}-1})\backslash \Gamma_0(Np)}f_0'\left|_{\kappa}\left[\begin{array}{cc}
0 & \epsilon\\
Np^{2^{2n}-1} & 0
\end{array}\right]\gamma\left[\begin{array}{cc}
0 & \epsilon\\
Np & 0
\end{array}\right]\right..
\]
Then $f_0\in S_{\kappa}(\Gamma_0(Np))$ and $f_0|U(p)=\alpha_{0,p}f_0$.We call such $f_0$ the $p$-stabilisation of $f$. The eigenform $f$ is called $p$-ordinary if $|\alpha_{0,p}|_p=1$. 

\begin{rem}
The construction of $f_0$ above is not necessarily nonzero as pointed out in \cite{CP}. In the sequel we shall always apply the assumption that $f_0\neq 0$. There is another construction of $p$-stabilisation in \cite{BS}. We do not use the method there because we do not have the `Andrianov-identity' (\cite[(9.6)]{BS}) in our setting. 
\end{rem}

\section{Eisenstein Series and the Doubling Method}
\label{section 3}

In this section, we use the doubling method to study $L$-functions. We define a Siegel Eisenstein series on $G_{2n}$ and consider its pull-back along $G_n\times G_n\to G_{2n}$. Its Petersson inner product against our eigenform gives an integral representation for the $L$-functions twisted by a Hecke character of $p$-power conductor. The Eisenstein series is defined such that its Fourier expansion has a natural $p$-adic interpolation. In particular, we use a different construction when we twist the $L$-function by a trivial character. This will result in the $p$-modification when constructing $p$-adic measures.

\subsection{Siegel Eisenstein series and its Fourier expansion}

We start by reviewing the Eisenstein series on $G_n(\mathbb{A})$ defined in \cite{Q,Sh99}. Denote $P_{n}\subset G_{n}$ be the Siegel parabolic subgroup consisting elements of the form $\left[\begin{array}{cc}
\ast & \ast\\
0_n & \ast
\end{array}\right]$. Fix $\kappa\in\Z$ and an odd prime $p$ splits in $\mathbb{B}$. Let $\chi$ be a Hecke character of conductor $p^{\mathfrak{c}}$ and assume $\chi_{\infty}(-1)=(-1)^{\kappa}$ in Case II. In Case I we do not need this assumption since $\det(g)\in\R_+$ for $g\in\mathrm{GL}_n(\mathbb{H})$. Let $\mathfrak{m}=(M)$ be an integral ideal of $\mathcal{O}$ such that $v_p(M)=2\mathfrak{c}$.  Let $\mu(g,s):=\prod_v\mu_v(g,s):G_{n}(\mathbb{A})\times\C\to\C$ be a function such that:
\begin{itemize}
\item{For the archimedean place, we set
\[
\mu_{\infty}(g,s)=j(g_{\infty},z_0)^{-\kappa}|j(g_{\infty},z_0)|_{\infty}^{\kappa-s}.
\]}

\item{For $v\nmid\mathfrak{m}$, we write $g=qk$ with $q=\left[\begin{array}{cc}
\hat{d}_q & b_q\\
0_n & d_q
\end{array}\right]\in P_{n}(\Q_v),k\in G_{n}(\mathcal{O}_v)$ by the Iwasawa decomposition and set 
\[
\mu_v(g,s)=\chi_v(\det d_q)^{-1}|\det d_q|_v^{-s}.
\]}

\item{For $v|\mathfrak{m}$, set 
\[
\mu_v(g,s)=\chi_v(\det d_q)^{-1}\chi_v(\det d_k)|\det d_q|_v^{-s}
\] 
if $g=qJk$ with $q=\left[\begin{array}{cc}
\hat{d}_q & b_q\\
0_n & d_q
\end{array}\right]\in P_{n}(\Q_v),k=\left[\begin{array}{cc}
a_k & b_k\\
c_k & d_k
\end{array}\right]\in K_0(\mathfrak{m})\subset G_{n}(\mathcal{O}_v)$ and $\mu_v(g,s)=0$ if $g$ is not of this form.}
\end{itemize}

We define a Siegel Eisenstein series on $G_{n}(\mathbb{A})$ by
\begin{equation}
\label{Eisenstein}
\mathbf{E}(g,s):=\mathbf{E}_{\kappa}(g,s;\chi,\mathfrak{m})=\sum_{\gamma\in P_{n}\backslash G_{n}}\mu(\gamma g,s).
\end{equation}
Clearly if we set $E(z,s)=j(g_z,z_0)^{\kappa}\mathbf{E}(g_z,s)$ for $z=g_z\cdot z_0$ then
\begin{equation}
E(z,s)=E_{\kappa}(z,s;\chi,\mathfrak{m})=\sum_{\gamma\in P_{n}\cap J\Gamma_0(M)J^{-1}\backslash J\Gamma_0(M)}\chi_{\mathfrak{m}}(\det c_{\gamma})^{-1}\delta(z)^{s-\kappa}|_{\kappa}\gamma.
\end{equation}

It converges absolutely when $\kappa>2n-1$ in Case I and $\kappa>2n+1$ in Case II. The Fourier expansions of $E(z,s)$ are discussed in \cite[Section 3]{Q},\cite[Section 5]{Sh99} and the special values of those confluent hypergeometric functions are calculated in \cite{Sh82}. We need to assume that all $l\nmid M$ splits in $\mathbb{B}$ so that the Eisenstein series has a nice Fourier expansion. We are interested in the special values at $s=\kappa$, write
\begin{equation}
\Lambda^n_M(\kappa,\chi)E_{\kappa}(z,\kappa,\chi)=\sum_{h\in\Lambda_n}a(h,y,\kappa,\chi)e(\lambda(hz)),\label{10}
\end{equation}
with $\Lambda_n=\{\tau\in S_n:\lambda(\tau s)\in\Z\text{ for any }s\in S(\mathcal{O})\}$. When $h>0$ we have\\
(Case I)
\begin{equation}
\begin{aligned}
a(h,y,\kappa,\chi)&=A_{\kappa}(n)\det(h)^{\kappa-\frac{2n-1}{2}}\prod_{l\nmid M}P_l(\chi^{\ast}(l)l^{-\kappa}),\\
A_{s}(n)&=2^{2-2n}(2\pi i)^{ns}(4D_{\mathbb{B}})^{-\frac{n(n-1)}{2}}\Gamma_n(2s)^{-1},\\
\Lambda_M^n(s,\chi)&=\prod_{i=0}^{n-1}L_M(2s-2i,\chi^2),\quad\Gamma_n(s)=\pi^{n(n-1)}\prod_{i=0}^{n-1}\Gamma(s-2i).\label{11}
\end{aligned}
\end{equation}
(Case II)
\begin{equation}
\begin{aligned}
a(h,y,\kappa,\chi)&=A_{\kappa}(n)\det(h)^{\kappa-\frac{2n+1}{2}}L_M(\kappa-n,\chi\rho_h)\prod_{l\nmid M}P_l(\chi^{\ast}(l)l^{-\kappa}),\\
A_{s}(n)&=2^{1-\frac{2n+1}{2}}(2\pi i)^{2ns}(D_{\mathbb{B}})^{-\frac{n(n+1)}{2}}\Gamma_{2n}(s)^{-1},\\
\Lambda_M^n(s,\chi)&=L_M(s,\chi)\prod_{i=1}^nL_M(2s-2i,\chi^2),\quad\Gamma_n(s)=\pi^{\frac{n(n-1)}{4}}\prod_{i=0}^{n-1}\Gamma(s-i/2).\label{12}
\end{aligned}
\end{equation}
Here $D_{\mathbb{B}}$ is the discriminant of $\mathbb{B}$, i.e. the product of non-split primes; $P_l\in\mathcal{O}[X]$ are polynomials such that for all primes $l$, $P_l=1$ if $\det(2h)\in\Z_l^{\times}$; $\rho_h$ is the character corresponding to $\Q(c^{1/2})/\Q$ with $c=(-1)^n\det(h)$. 

\begin{prop}
\label{3.1}
Keep the notation as above and $\sigma\in\mathrm{Aut}(\C)$. Assume $h>0$ then\\
(1) (Case I)
\begin{equation}
\left(\frac{a(h,y,\kappa,\chi)}{A_{\kappa}(n)}\right)^{\sigma}=\frac{a(h,y,\kappa,\chi^{\sigma})}{A_\kappa(n)}.
\end{equation}
(2) (Case II)
\begin{equation}
\left(\frac{a(h,y,\kappa,\chi)}{A_{\kappa}(n)(\pi i)^{\kappa-n}g(\chi\rho_h)}\right)^{\sigma}=\frac{a(h,y,\kappa,\chi^{\sigma})}{A_{\kappa}(n)(\pi i)^{\kappa-n}g(\chi^{\sigma}\rho^{\sigma}_h)}.
\end{equation}
Here $g(\chi)$ is the Gauss sum for $\chi$.
\end{prop}

\begin{proof}
The proposition follows from the explicit formulas for $a(h,y,\kappa,\chi)$. For Case II we are using the fact that
\[
\left(\frac{L(k,\chi)}{(\pi i)^kg(\chi)}\right)^{\sigma}=\frac{L(k,\chi^{\sigma})}{(\pi i)^kg(\chi^{\sigma})}.
\]
\end{proof}

\subsection{The integral representation: nontrivial twist}

We consider the following embedding
\begin{equation}
\label{doublinggroup}
\begin{aligned}
G_n\times G_n&\to G_{2n},\\
\left[\begin{array}{cc}
a_1 & b_1\\
c_1 & d_1
\end{array}\right]\times\left[\begin{array}{cc}
a_2 & b_2\\
c_2 & d_2
\end{array}\right]&\mapsto\left[\begin{array}{cccc}
a_1 & 0 & b_1 & 0\\
0 & a_2 & 0 & b_2\\
c_1 & 0 & d_1 & 0\\
0 & c_2 & 0 & d_2
\end{array}\right].
\end{aligned}
\end{equation}
We then view $G_n\times G_n$ as a subgroup of $G_{2n}$ and denote $\gamma_1\times\gamma_2$ as the image of $(\gamma_1,\gamma_2)$ in $G_{2n}$ under the embedding. We can also define an embedding for symmetric spaces
\begin{equation}
\label{doublingdomain}
\mathcal{H}_n\times\mathcal{H}_n\to\mathcal{H}_{2n},\quad z\times w\mapsto \left[\begin{array}{cc}
z & 0\\
0 & w
\end{array}\right].
\end{equation}
This embedding is compatible with the embedding of groups in the sense that $(\gamma_1\times\gamma_2)(z\times w)=\gamma_1z\times\gamma_2w$. To ease the notation, we keep the subscript $2n$ and omit the subscript $n$ if it is clear from the context.

Let $\mathbf{f}\in S_{\kappa}(K_0(\mathfrak{n}p))$ be an eigenform with $(\mathfrak{n},p)=1$ and $\chi$ is a nontrivial Hecke character with conductor $p^{\mathfrak{c}}>1$. We are now going to define a function on $G(\mathbb{A})\times G(\mathbb{A})$ (for $s\in\C$) twisting the Siegel Eisenstein series on $G_{2n}(\mathbb{A})$ defined in \eqref{Eisenstein} . Take $\mathfrak{m}=\mathfrak{n}^2p^{2\mathfrak{c}}$ in the last subsection and define
\begin{equation}
\mathcal{E}(g,h,s):=\mathcal{E}_{\kappa}(g,h,s;\chi,\mathfrak{n}^2p^{2\mathfrak{c}})=\sum_{\gamma\in P_{2n}\backslash G_{2n}}\tilde{\mu}(\gamma(g\times h),s),
\end{equation}
where:
\begin{itemize}
\item{$\tilde{\mu}_v(g,s)=\mu_v(g,s)$ for $v\nmid\mathfrak{n}p$,}

\item{For $v|\mathfrak{n}$, set 
\begin{equation}
\tilde{\mu}_v(g,s)=\mu_v\left(g\left[\begin{array}{cccc}
1_n & 0 & 0 & |\mathfrak{n}|_v\\
0 & 1_n & -\epsilon|\mathfrak{n}|_v & 0\\
0 & 0 & 1_n & 0\\
0 & 0 & 0 & 1_n
\end{array}\right],s\right),
\end{equation}}

\item{For $v=p$, set
\begin{equation}
\tilde{\mu}_p(g,s)=\sum_{x\in\mathrm{GL}_n(\mathcal{O}_p)/p^{\mathfrak{c}}\mathrm{GL}_n(\mathcal{O}_p)}\chi_p(\det x)^{-1}\mu_p\left(g\left[\begin{array}{cccc}
1_n & 0 & 0 & \frac{x}{p^{\mathfrak{c}}}\\
0 & 1_n & \frac{-\epsilon x^{\ast}}{p^{\mathfrak{c}}} & 0\\
0 & 0 & 1_n & 0\\
0 & 0 & 0 & 1_n
\end{array}\right],s\right).
\end{equation}}
\end{itemize}
Then $\mathcal{E}(g,h,s)\in\mathcal{M}_{\kappa}(K_0(\mathfrak{n}^2p^{2\mathfrak{c}}))\otimes\mathcal{M}_{\kappa}(K_0(\mathfrak{n}^2p^{2\mathfrak{c}}))$ (holomorphic at $s=\kappa$). That is
\[
\mathcal{E}(gk_1,hk_2,s)=\mathcal{E}(g,h,s)\text{ for }k_1,k_2\in K_0(\mathfrak{n}^2p^{2\mathfrak{c}}).
\]

\begin{rem}
In next section, we will define the classical counterpart $\mathfrak{E}(z,w,s)$ of $\mathcal{E}_{\kappa}(g,h,s)$. One sees that the function $\mathcal{E}_{\kappa}$ constructed here is same as the twisted Eisenstein series constructed in \cite[Section 2]{BS} except the sections at $v|\mathfrak{n}$. At such places, both in \cite{BS} and here the local sections $\tilde{\mu}_v$ are chosen such that it only contributes constants in inner product formula, i.e. all Euler factors at $v|\mathfrak{n}$ are removed (see lemma \ref{3.5}). Hence the inner product formula obtained here in Proposition \ref{prop 3.6}, \ref{prop 3.8} is same as the one in \cite{BS} up to constants.
\end{rem}

Denote $h^c=\left[\begin{array}{cc}
-1_n & 0\\
0 & 1_n
\end{array}\right]h\left[\begin{array}{cc}
-1_n & 0\\
0 & 1_n
\end{array}\right]$ and consider the integral
\begin{equation}
\int_{G(\Q)\backslash G(\mathbb{A})}\mathcal{E}(g,h^c,s)\overline{\mathbf{f}(g)}\mathbf{d}g.\label{22}
\end{equation}
We compute this integral using the well known `doubling method' strategy (see for example \cite{Sh95,Sh00}). That is the function $\mathcal{E}$ can be decomposed as a sum according to the coset decomposition of $P_{2n}(\Q)\backslash G_{2n}(\Q)/G(\Q)\times G(\Q)$. By the cuspidality of $\mathbf{f}$, only the coset given by
\[
\tau=\left[\begin{array}{cccc}
1 & 0 & 0 & 0\\
0 & 1 & 0 & 0\\
0 & 1 & 1 & 0\\
-\epsilon & 0 & 0 & 1
\end{array}\right],
\]
with $\epsilon=-1$ in Case I and $\epsilon=1$ in Case II can contribute to above integral \eqref{22}. In this case we obtain
\[
\eqref{22}=\int_{G(\mathbb{A})}\tilde{\mu}\left(\tau\left(g\times \left[\begin{array}{cc}
0 & \epsilon\\
1 & 0
\end{array}\right]\right),s\right)\overline{\mathbf{f}(hg)}\mathbf{d}g.
\]
We discuss the above integral locally case by case in following lemmas. 

\begin{lem}
Assume $\mathrm{Re}(s)> 4n+1-\kappa$ then
\begin{equation}
\int_{G(\R)}\tilde{\mu}_{\infty}\left(\tau\left(g\times \left[\begin{array}{cc}
0 & \epsilon\\
1 & 0
\end{array}\right]\right),s\right)\overline{\mathbf{f}(hg)}\mathbf{d}g=c_{\kappa}(s)\overline{\mathbf{f}(h)}\label{23}
\end{equation}
with\\
(Case I)
\[
c_{\kappa}(s)=\pi^{n(2n-1)}\prod_{i=0}^{n-1}\frac{\Gamma(s+\kappa-2i)}{\Gamma(s+\kappa+n-2i)},
\]
(Case II)
\[
c_{\kappa}(s)=\pi^{n(2n+1)}\prod_{i=0}^{2n-1}\frac{\Gamma\left(\frac{s+\kappa-i}{2}\right)}{\Gamma\left(\frac{s+\kappa+2n-i}{2}\right)}.
\]
\end{lem}

\begin{proof}
Denote $z=g\cdot z_0$ and $f(z)=j(g,z_0)^{\kappa}\mathbf{f}(g)$. The integral in \eqref{23} can be rewritten as
\[
\int_{\mathcal{H}}\det(z+z_0)^{-\kappa}|\det(z+z_0)|_{\infty}^{\kappa-s}\delta(z)^{\kappa+s}\overline{j(h,z)^{-\kappa}f(hz)}\mathbf{d}z.
\]
This kind of integral is calculated by the same method as \cite[Propostion A.2.9]{Sh97}. We remind the reader that our function $\delta$ here is actually $\delta^{1/2}$ there. There exists a constant $c_{\kappa}(s)$ independent of $f$ such that above integral equals 
\[
c_{\kappa}(s)\overline{j(h,z_0)^{-\kappa}f(hz_0)}=c_{\kappa}(s)\overline{\mathbf{f}(h)}.
\]
The constant $c_{\kappa}(s)$ is simply determined by taking $z=0,f=1$. That is 
\[
c_{\kappa}(s)=\int_{\mathcal{H}}\delta(z)^{{\kappa}+s}\mathbf{d}z=\int_{\mathcal{H}}\delta(z)^{\kappa+s-2n-\epsilon}dz.
\]
By changing our symmetric space to bounded symmetric space by Cayley transform, this constant $c_{\kappa}(s)$ is calculated in \cite[Theorem 2.4.1]{Hua} for Case I and \cite[Theorem 2.3.1]{Hua} for Case II. 
\end{proof}

To ease the notation, we only discuss Case I in following two lemmas and omit the same proof for Case II.

\begin{lem}
Let $v\nmid\mathfrak{m}$ corresponds to a prime $l$, then
\begin{equation}
\int_{G(\Q_l)}\mu_l(\tau(g\times J),s)\overline{\mathbf{f}(hg)}\mathbf{d}g=\Lambda^{2n}_{l}(s,\chi)^{-1}L^{(l)}(s,\mathbf{f},\chi)\overline{\mathbf{f}(h)}.\label{24}
\end{equation}
\end{lem}

\begin{proof}
Recall the Cartan decomposition
\[
G(\Q_l)=\coprod_{\substack{e_1,...,e_n\in\Z\\
0\leq e_1\leq...\leq e_n}}G(\mathcal{O}_l)\mathrm{diag}[k^{-1}_{e_1,...,e_n},k_{e_1,...,e_n}]G(\mathcal{O}_l),k_{e_1,...,e_n}=\mathrm{diag}[l^{e_1},...,l^{e}_n].
\]
Denote the Hecke operator corresponds to $\mathrm{diag}[k^{-1}_{e_1,...,e_n},k_{e_1,...,e_n}]$ by $T_{e_1,...,e_n}$. The integral on the left hand side of \eqref{24} can be written as
\[
\begin{aligned}
&\sum_{\substack{e_1,...,e_n\in\Z\\
0\leq e_1\leq...\leq e_n}}\mu_l(\tau(\mathrm{diag}[k_{e_1,...,e_n}^{-1},k_{e_1,...,e_n}]\times J))\sum_{k\in G(\mathcal{O}_v)\mathrm{diag}[k_{e_1,...,e_n}^{-1},k_{e_1,...,e_n}]G(\mathcal{O}_v)/G(\mathcal{O}_v)}\overline{\mathbf{f}(hk)}\\
=&\sum_{\substack{e_1,...,e_n\in\Z\\
0\leq e_1\leq...\leq e_n}}\chi_l(\det k_{e_1,...,e_n})\det(k_{e_1,...,e_n})^{-s}\overline{T_{e_1,...,e_n}\mathbf{f}(h)}\\
=&\Lambda^{2n}_{v}(s,\chi)^{-1}L^{(l)}(s,\mathbf{f},\chi)\overline{\mathbf{f}(h)}.
\end{aligned}
\]

\end{proof}

\begin{lem}
\label{3.5}
$\text{ }$\\
(1) Denote $\Q_{\mathfrak{n}}=\prod_{v|\mathfrak{n}}\Q_v$ and $\mathcal{O}_{\mathfrak{n}}=\prod_{v|\mathfrak{n}}\mathcal{O}_v$ then
\begin{equation}
\int_{G(\Q_{\mathfrak{n}})}\tilde{\mu}_{\mathfrak{n}}(\tau(g\times J),s)\overline{\mathbf{f}(hg)}\mathbf{d}g=\chi_{\mathfrak{n}}(\mathfrak{n})^{-2}\mathfrak{n}^{2ns}\mathrm{vol}(K^1_{\mathfrak{n}})\overline{\mathbf{f}(h\mathrm{diag}[\mathfrak{n}^{-1},\mathfrak{n}]J)}.\label{25}
\end{equation}
Here, for any integer $\mathfrak{n}$ we denote
\[
K^1_{\mathfrak{n}}=\left\{\gamma=\left[\begin{array}{cc}
a & b\\
c & d
\end{array}\right]\in G(\mathcal{O}_{\mathfrak{n}}):a,d\equiv 1\text{ mod }\mathfrak{n},c\equiv0\text{ mod }\mathfrak{n}^2\right\}.
\]
(2) 
\begin{equation}
\begin{aligned}
&\int_{G(\Q_p)}\tilde{\mu}_p(\tau(g\times J),s)\overline{\mathbf{f}(hg)}\mathbf{d}g\\
=&p^{2\mathfrak{c}ns}\mathrm{vol}(\mathrm{GL}_n(\mathcal{O}_p)/p^{\mathfrak{c}}\mathrm{GL}_n(\mathcal{O}_p))\mathrm{vol}(K_{p^{\mathfrak{c}}}^1)\overline{\mathbf{f}(h\mathrm{diag}[p^{-\mathfrak{c}},p^{\mathfrak{c}}]J)}.\label{26}
\end{aligned}
\end{equation}
\end{lem}

\begin{proof}
Note that
\[
\tilde{\mu}_{\mathfrak{n}}(\tau (g\times J),s)=\mu_{\mathfrak{n}}(\tau(gJ^{-1}\mathrm{diag}[\mathfrak{n},\mathfrak{n}^{-1}]\times 1)\tau^{-1}d_{\mathfrak{n}}^{-1}J,s)
\]
with $d_{\mathfrak{n}}=\mathrm{diag}[\mathfrak{n}\cdot 1_n,1_n,\mathfrak{n}^{-1}\cdot 1_n,1_n]$. Changing variables of the integral in \eqref{25} and consider
\[
\int_{G(\Q_{\mathfrak{n}})}\mu(\tau(g\times 1)\tau^{-1}d_{\mathfrak{n}}^{-1}J,s)\overline{\mathbf{f}(hg\mathrm{diag}[\mathfrak{n}^{-1},\mathfrak{n}]J)}\mathbf{d}g.
\]
Let $g=\left[\begin{array}{cc}
a & b\\
c & d
\end{array}\right]$, then 
\[
\tau(g\times 1)\tau^{-1}d_{\mathfrak{n}}^{-1}J=\left[\begin{array}{cccc}
b\mathfrak{n} & 0 & -a\mathfrak{n}^{-1} & b\\
0 & 0 & 0 & -1\\
d\mathfrak{n} & 0 & -c\mathfrak{n}^{-1} & d-1\\
b\mathfrak{n} & 1 & -(a-1)\mathfrak{n}^{-1} & b
\end{array}\right].
\]
By the definition of $\mu$, the integrand is nonzero unless $d$ is invertible and
\[
\left[\begin{array}{cc}
d\mathfrak{n} & 0\\
b\mathfrak{n} & 1
\end{array}\right]^{-1}\left[\begin{array}{cc}
-c\mathfrak{n}^{-1} & d-1\\
(-a+1)\mathfrak{n}^{-1} & b
\end{array}\right]\in M_{2n}(\mathcal{O}_{\mathfrak{n}}),
\]
which implies $g\in K^1_{\mathfrak{n}}$. Hence the integral equals
\[
\int_{K^1_{\mathfrak{n}}}\chi_{\mathfrak{n}}(\mathfrak{n})^{-2}\mathfrak{n}^{2ns}\overline{\mathbf{f}(hg\mathrm{diag}[\mathfrak{n}^{-1},\mathfrak{n}]J)}\mathbf{d}g=\chi_{\mathfrak{n}}(\mathfrak{n})^{-2}\mathfrak{n}^{2ns}\mathrm{vol}(K^1_{\mathfrak{n}})\overline{\mathbf{f}(h\mathrm{diag}[\mathfrak{n}^{-1},\mathfrak{n}]J)}.
\]
For (2) note that
\[
\tilde{\mu}_p(\tau(g\times 1),s)=\sum_{x\in\mathrm{GL}_n(\mathcal{O}_p)/p^{\mathfrak{c}}\mathrm{GL}_n(\mathcal{O}_p)}\chi(\det x)^{-1}\mu_p(\tau(gJ^{-1}\mathrm{diag}[p^{\mathfrak{c}},p^{-\mathfrak{c}}]\times 1)\tau^{-1}d_{p}^{-1}J,s),
\]
with $d_p=\mathrm{diag}[p^{\mathfrak{c}}x^{-1},1_n,p^{-\mathfrak{c}}x,1_n]$. Changing variables of the integral in \eqref{26} we obtain
\[
\begin{aligned}
&\sum_{x\in\mathrm{GL}_n(\mathcal{O}_p)/p^{\mathfrak{c}}\mathrm{GL}_n(\mathcal{O}_p)}\chi(\det x)^{-1}\int_{G(\Q_p)}\mu_p(\tau(g\times 1)\tau^{-1}d_{\mathfrak{n}}^{-1}J,s)\overline{\mathbf{f}(hg\mathrm{diag}[p^{-\mathfrak{c}},p^{\mathfrak{c}}]J)}\mathbf{d}g\\
=&\sum_{x\in\mathrm{GL}_n(\mathcal{O}_p)/p^{\mathfrak{c}}\mathrm{GL}_n(\mathcal{O}_p)}\chi(\det x)^{-1}\cdot p^{\mathfrak{c}ns}\mathrm{vol}(K_{p^{\mathfrak{c}}}^1)\chi(\det x)\overline{\mathbf{f}(h\mathrm{diag}[p^{-\mathfrak{c}},p^{\mathfrak{c}}]J)}\\
=&p^{2\mathfrak{c}ns}\mathrm{vol}(\mathrm{GL}_n(\mathcal{O}_p)/p^{\mathfrak{c}}\mathrm{GL}_n(\mathcal{O}_p))\mathrm{vol}(K_{p^{\mathfrak{c}}}^1)\overline{\mathbf{f}(h\mathrm{diag}[p^{-\mathfrak{c}},p^{\mathfrak{c}}]J)}.
\end{aligned}
\]
\end{proof}

Denote 
\[
\Omega^{\mathfrak{c}}_{\kappa}(s)=\mathfrak{n}^{2ns}p^{2\mathfrak{c}ns}\chi_{\mathfrak{n}}(\mathfrak{n})^{-2}\mathrm{vol}(K_{\mathfrak{np^{\mathfrak{c}}}}^1)\mathrm{vol}(\mathrm{GL}_n(\mathcal{O}_p)/p^{\mathfrak{c}}\mathrm{GL}_n(\mathcal{O}_p))c_{\kappa}(s)
\]
for the product of all constants in above lemmas. We summarize above discussions in the following proposition.

\begin{prop}
Let $\mathbf{f}\in S_{\kappa}(K_0(\mathfrak{n}p))$ be an eigenform with $(\mathfrak{n},p)=1$ and $\chi$ is a nontrivial Hecke character of conductor $p^{\mathfrak{c}}>1$. Then
\begin{equation}
\int_{G(\Q)\backslash G(\mathbb{A})}\mathcal{E}(g,h^c,s)\overline{\mathbf{f}\left(g\left[\begin{array}{cc}
0 & \epsilon\mathfrak{n}^{-1}p^{-\mathfrak{c}}\\
\mathfrak{n}p^{\mathfrak{c}} & 0
\end{array}\right]\right)}\mathbf{d}g=\frac{\Omega_{\kappa}^{\mathfrak{c}}(s)}{\Lambda^{2n}_{\mathfrak{n}p}(s,\chi)}L_{\mathfrak{np}}(s,\mathbf{f},\chi)\overline{\mathbf{f}\left(h\right)}.
\label{27}
\end{equation}
\label{prop 3.6}
\end{prop}

\subsection{The modification at $p$}

Recall that $p$ is a fixed odd prime split in $\mathbb{B}$, say $p\mathcal{O}=\mathfrak{pp}'$. Denote $\varpi$ for a uniformizer of $\mathfrak{p}$. Let $\mathbf{f}\in S_{\kappa}(K_0(\mathfrak{n}p))$ be an eigenform with $(\mathfrak{n},p)=1$ and $\chi=1$ is the trivial character. 

For $0\leq i\leq n$ and $r=0$ or $1$ denote $T(p_{i,r})$ for the Hecke operator given by double coset
\[
[K_0(p)\xi K_0(p)],\quad\xi=\mathrm{diag}[\hat{w},w],\quad w=\left[\begin{array}{ccc}
1_{n-i-1} & 0 & 0\\
0 & {\varpi}^r & 0\\
0 & 0 & p\cdot 1_i
\end{array}\right].
\]
Suppose there is a double coset decomposition
\[
\mathrm{GL}_n(\mathcal{O}_p)w\mathrm{GL}_n(\mathcal{O}_p)=\coprod_j\mathrm{GL}_n(\mathcal{O}_p)\delta_{rij}.
\]
Define 
\begin{equation}
\mathcal{E}(g,h,s,i,r)=\sum_{\gamma\in P_{2n}\backslash G_{2n}}\mu^{i,r}(\gamma(g\times h),s)
\end{equation}
with $\mu_v^{i,r}=\tilde{\mu}_v$ as in the last subsection for $v\neq p$ and
\begin{equation}
\mu_p^{i,r}(g,s)=\sum_j\sum_{x\in pM_n(\mathcal{O}_p)\hat{\delta}_{rij}/pM_n(\mathcal{O}_p)}\mu_p\left(g\left[\begin{array}{cccc}
1_n & 0 & 0 & \frac{x}{p}\\
0 & 1_n & -\frac{\epsilon x^{\ast}}{p} & 0\\
0 & 0 & 1_n & 0\\
0 & 0 & 0 & 1_n
\end{array}\right],s\right).
\end{equation}

\begin{lem}
Let $\lambda_{2i+r}$ be the Hecke eigenvalues of $\mathbf{f}$ under $T(p_{i,r})$, i.e. $T(p_{i,r})\mathbf{f}=\lambda_{2i+r}\mathbf{f}$. Then
\begin{equation}
\begin{aligned}
&\int_{G(\Q_p)}\mu_p^{i,r}\left(\tau\left(g\times\left[\begin{array}{cc}
0 & \epsilon\\
1 & 0
\end{array}\right]\right),s\right)\overline{\mathbf{f}\left(hg\left[\begin{array}{cc}
0 & \epsilon p^{-1}\\
p & 0
\end{array}\right]\right)}\mathbf{d}g\\
=&p^{s(2i+r)}L^{(p)}(s,\mathbf{f})\lambda_{2n-2i-r}\overline{\mathbf{f}\left(h\right)}.
\end{aligned}
\end{equation}
\end{lem}

\begin{proof}
Again we only prove for Case I and omit the same proof for Case II. Note that
\[
\left[\begin{array}{cccc}
1 & 0 & 0 & \frac{x}{p}\\
0 & 1 & \frac{x^{\ast}}{p} & 0\\
0 & 0 & 1 & 0\\
0 & 0 & 0 & 1
\end{array}\right]=J^{-1}d_p\tau_x^{-1}d_p^{-1}J,
\]
with
\[
\tau_x=\left[\begin{array}{cccc}
1 & 0 & 0 & 0\\
0 & 1 & 0 & 0\\
0 & x & 1 & 0\\
x^{\ast} & 0 & 0 & 1 
\end{array}\right],d_p=\left[\begin{array}{cccc}
p & 0 & 0 & 0\\
0 & 1 & 0 & 0\\
0 & 0 & p^{-1} & 0\\
0 & 0 & 0 & 1
\end{array}\right].
\]
We need to compute
\begin{equation}
\int_{G(\Q_p)}\sum_{j,x}\mu_p(\tau(g\times 1)\tau_x^{-1}d_x^{-1}J)\overline{\mathbf{f}(hg)}\mathbf{d}g.\label{30}
\end{equation}
Let $g=\left[\begin{array}{cc}
a & b\\
c & d
\end{array}\right]$, then
\[
\tau(g\times 1)\tau_x^{-1}d_x^{-1}J=\left[\begin{array}{cccc}
pb & 0 & -p^{-1}a & bx\\
0 & 0 & 0 & -1\\
pd & 0 & -p^{-1}c & dx-1\\
pb & 1 & -p^{-1}(a-x^{\ast}) & bx
\end{array}\right].
\]
We need $d$ to be invertible and
\[
\left[\begin{array}{cc}
pd & 0\\
pb & 1
\end{array}\right]^{-1}\left[\begin{array}{cc}
-p^{-1}c & dx^{\ast}-1\\
-p^{-1}(a-x) & bx^{\ast}
\end{array}\right]=\left[\begin{array}{cc}
-p^{-2}d^{-1}c & p^{-1}x-p^{-1}d^{-1}\\
p^{-1}x^{\ast}-p^{-1}\hat{d} & bd^{-1}
\end{array}\right]\in M_{2n}(\mathcal{O}_p).
\]
That is
\[
bd^{-1}\in M_{n}(\mathcal{O}_p),d^{-1}c\in p^2M_n(\mathcal{O}_p),d^{-1}\in x+pM_n(\mathcal{O}_p).
\]
There is a permutation $j\mapsto j'$ such that $\hat{d}$ runs through $\mathrm{GL}_n(\mathbb{B}_p)\cap pM_n(\mathcal{O}_p)\delta_{rij'}^{-1}$ for fixed $j$. Hence, taking the sum over $j,x$, \eqref{30} can be written as
\[
\sum_{g}|\det(pd)|_p^{-s}\overline{\mathbf{f}(hg)}\mathbf{d}g.
\]
with the sum taking over 
\[
g\in K_0(p^2)\mathrm{diag}[\hat{d},d]K_0(p^2)\text{ with }\hat{d}\in\mathrm{GL}_n(\mathbb{B}_p)\cap p M_n(\mathcal{O}_p)w^{-1}.
\]
Equivalently, we take the sum of $g$ over 
\[
K_0(p^2)\mathrm{diag}[\hat{d},d]K_0(p^2)\cdot K_0(p^2)\mathrm{diag}[p\hat{w},p^{-1}w]K_0(p^2)\text{ with }\hat{d}\in\mathrm{GL}_n(\mathbb{B}_p)\cap M_n(\mathcal{O}_p).
\]
Let $T_d$ be the Hecke operator given by the double coset $K_0(p^2)\mathrm{diag}[d^{\ast},d^{-1}]K_0(p^2)$ and $T_w$ given by the double coset $K_0(p^2)\mathrm{diag}[p^{-1}w^{\ast},pw^{-1}]K_0(p^2)$. Then \eqref{30} equals
\[
\begin{aligned}
&\sum_{\substack{\xi\in K_0(p^2)\backslash\mathfrak{X}/K_0(p^2)\\\xi=\mathrm{diag}[d^{\ast},d^{-1}]}}|\det(d)|_p^{-s}|\det(w)|_p^{-s}\overline{\mathbf{f}|T_d|T_w(h)}\\
=&L^{(p)}(s,\mathbf{f})|\det(w)|_p^{-s}\lambda_{2n-2i-r}\overline{\mathbf{f}(h)}
\end{aligned}
\]
as desired.
\end{proof}

Combining the computations for integrals outside $p$ in last subsection we obtain the following proposition.

\begin{prop}
Let $\mathbf{f}\in S_{\kappa}(K_0(\mathfrak{n}p))$ be an eigenform with $(\mathfrak{n},p)=1$. Write $\beta_{1,p},...,\beta_{2n,p}$ be its Satake parameters at $p$. Then
\begin{equation}
\label{32}
\begin{aligned}
&\int_{G(\Q)\backslash G(\mathbb{A})}\sum_{\substack{0\leq i\leq n\\r=0,1}}(-1)^rp^{\frac{(2i+r)(2i+r-1)}{2}-2n(2i+r)}\mathcal{E}(g,h^c,s,i,r)\overline{\mathbf{f}\left(g\left[\begin{array}{cc}
0 & \epsilon\mathfrak{n}^{-1}p^{-1}\\
\mathfrak{n}p &0
\end{array}\right]\right)}\mathbf{d}g\\
=&\frac{p^{-2n^2-2n+2ns}\Omega^0_{\kappa}(s)}{\Lambda_{\mathfrak{n}p}^{2n}(s)}\prod_{i=1}^{2n}(1-\beta_{i,p}p^{-s+2n+\epsilon})L_{\mathfrak{n}}(s,\mathbf{f})\overline{\mathbf{f}(h)}.
\end{aligned}
\end{equation}
\label{prop 3.8}
\end{prop}

\begin{proof}
The integral can be written as
\[
\begin{aligned}
&\frac{\Omega_{\kappa}(s)}{\Lambda_{\mathfrak{n}p}^{2n}(s)}\left(\sum_{i,r}(-1)^rp^{\frac{(2i+r)(2i+r-1)}{2}-2n(2i+r)}\lambda_{2n-2i-r}(p^{-s})^{2n-2i-r}\right)L_{\mathfrak{n}}(s,\mathbf{f})\overline{\mathbf{f}(h)}\\
=&\frac{p^{-2n^2-2n+2ns}\Omega^0_{\kappa}(s)}{\Lambda_{\mathfrak{n}p}^{2n}(s)}\left(\sum_{i=0}^{2n}(-1)^ip^{\frac{i(i-1)}{2}}\lambda_i(p^{-s+1})^i\right)L_{\mathfrak{n}}(s,\mathbf{f})\overline{\mathbf{f}(h)}.
\end{aligned}
\]
The proposition then follows by following two `Tamagawa-identities' relating the eigenvalues and Satake parameters.\\
(Case I)
\[
\sum_{i=0}^{2n}(-1)^ip^{\frac{i(i-1)}{2}}\lambda_iX^i=\prod_{i=1}^{2n}(1-\beta_{i,p}p^{2n-2}X),
\]
(Case II)
\[
\sum_{i=0}^{2n}(-1)^ip^{\frac{i(i-1)}{2}}\lambda_iX^i=\prod_{i=1}^{2n}(1-\beta_{i,p}p^{2n}X).
\]
Note that if we identify $G_n(\Q_v)$ with orthogonal group $\mathrm{SO}(2n,2n)(\Q_v)$ in Case I and symplectic group $\mathrm{Sp}(2n,\Q_v)$ in Case II then the Hecke operator $T(p_{i,r})$ can be identified with the one for $\mathrm{SO}(2n,2n)(\Q_v)$ or $\mathrm{Sp}(2n,\Q_v)$ given by double coset
\[
[K_0(p)\xi K_0(p)],\xi=\mathrm{diag}[\hat{w},w],w=\left[\begin{array}{cc}
1_{2n-2i-r} & 0\\
0 & p\cdot 1_{2i+r}
\end{array}\right].
\]
Here we are abusing the notation $K_0(p)$ for subgroup of $\mathrm{SO}(2n,2n)(\Z_v)$ or $\mathrm{Sp}(2n,\Z_v)$ consisting elements of the form $\left[\begin{array}{cc}
a &b\\
c & d
\end{array}\right]$ with $c\equiv 0\text{ mod }p$. Then above two identities can be proved by using \cite[Lemma 19.13]{Sh00} and explicit description of Satake map as chosen in \cite[Proposition 6.3]{Q} for Case I and \cite[Theorem 2.9]{Sh99} for Case II.
\end{proof}

\section{Differential Operators and Classical Reformulations}
\label{section 4}

From integral representation of the $L$-function in \eqref{27} and \eqref{32}, the Fourier expansion of Eisenstein series at special value $s=\kappa$ will provide information of special value of the $L$-function at $s=\kappa$. To obtain the information for more special values of the $L$-function, we will apply certain differential operators on the Eisenstein series to shift its weight. In this section, we construct certain differential operators associated to $G_{2n}(\R)$ which are compatible with the doubling embedding $G_n(\R)\times G_n(\R)\to G_{2n}(\R)$, generalizing the construction in \cite{B85, BS} for symplectic groups. After applying such differential operator on the Eisenstein series, we will also reformulate Proposition \ref{prop 3.6} and Proposition \ref{prop 3.8} in the classical setting in Proposition \ref{prop4.5}.

We first recall some results in \cite{B85,BS}. Let $\mathfrak{H}_n$ be the Siegel upper half plane with the action of $\mathrm{Sp}(n)(\R)$. There are compatible doubling maps
\[
\mathrm{Sp}(n)(\R)\times\mathrm{Sp}(n)(\R)\to\mathrm{Sp}(2n)(\R),\quad\mathfrak{H}_n\times\mathfrak{H}_n\to\mathfrak{H}_{2n}.
\]
For any $\alpha\in\C$ there is a differential operator $\mathcal{D}_{n,\alpha}$ such that for any integer $\kappa>n$ and any holomorphic function $f$ and $\gamma\in\mathrm{Sp}(n)(\R)$ we have
\begin{equation}
\label{33}
\begin{aligned}
\mathcal{D}_{n,\kappa}(f|V)&=(\mathcal{D}_{n,\kappa}f)|V,\\
\mathcal{D}_{n,\kappa}(f|_{\kappa}1\times\gamma)&=(\mathcal{D}_{n,\kappa}f)|_{\kappa+1}1\times\gamma,\\
\mathcal{D}_{n,\kappa}(f|_{\kappa}\gamma\times 1)&=(\mathcal{D}_{n,\kappa}f)|_{\kappa+1}\gamma\times 1.
\end{aligned}
\end{equation}
Here $V$ is the operator such that
\[
(f|V)\left(\left[\begin{array}{cc}
z_1 & z_2\\
z_3 & z_4
\end{array}\right]\right)=f\left(\left[\begin{array}{cc}
z_4 & z_3\\
z_2 & z_1
\end{array}\right]\right).
\]

Now let 
\[
G_n(\R)\times G_n(\R)\to G_{2n}(\R),\quad\mathcal{H}_n\times\mathcal{H}_n\to\mathcal{H}_{2n},
\]
be compatible doubling maps as in \eqref{doublinggroup} and \eqref{doublingdomain}. We are going to construct differential operators $\mathfrak{D}_{\alpha}:=\mathfrak{D}_{n,\alpha}$ on $\mathcal{H}_{2n}$ having similar transformation properties as \eqref{33}. In Case II we have $\mathcal{H}_{2n}\cong\mathfrak{H}_{4n}$ so the pullback of $\mathcal{D}_{2n,\alpha}$ along $\iota:\mathcal{H}_{2n}\stackrel{\sim}\longrightarrow\mathfrak{H}_{4n}$ has desired properties. We then focus on the Case I.

\subsection{Preliminaries}

We start by recalling some notations about multi-linear algebra used in \cite[III.6]{F}. Set $N=\{1,...,n\}$. For $0\leq p\leq n$, we denote
\[
N_p:=\{a=\{a_1<a_2<...<a_n\}\subset N\}
\]
for the set of all subsets of $N$ with $p$ elements. We view $M_{\left(\substack{n\\p}\right)}(\C)=\{(A_{a,b})_{a,b\in N_p}\}$ as a set of matrices whose entries are labeled by subsets $a,b\in N_p$ of $N$ with $p$ elements. Let $A\in M_n(\C)$ and for $0\leq h\leq n$, we define $A^{[p]}$ and $\mathrm{Ad}^{[p]}A$ in $M_{\left(\substack{n\\p}\right)}(\C)$ by setting their $(a,b)$-entries as
\[
\begin{aligned}
(A^{[p]})_{a,b}&=|A|^a_b:=\det((A_{i,j})_{i\in a,j\in b}),\\
(\mathrm{Ad}^{[p]}A)_{a,b}&=\epsilon(a,N\backslash a)\epsilon(b,N\backslash b)|A|^{N\backslash b}_{N\backslash a},
\end{aligned}
\]
where for two subsets $a'=\{a_1',...,a_p'\}$, $a''=\{a_1'',...,a_q''\}$ of $N$ with $a_1'<...<a_p'$ and $a_1''<...<a_q''$ we write $\epsilon(a',a'')$ for the sign of permutation which takes the $(p+q)$-tuple $\{a'_1,...,a_p',a_1'',...,a_p''\}$ into its natural order. For $0\leq p,q\leq n$, we define the $\sqcap$-multiplication by
\[
\begin{aligned}
\sqcap: M_{\left(\substack{n\\p}\right)}(\C)\times M_{\left(\substack{n\\q}\right)}(\C)&\to M_{\left(\substack{n\\p+q}\right)}(\C),\\
(A\sqcap B)_{a,b}&:=\frac{1}{\left(\substack{p+q\\p}\right)}\sum_{\substack{a=a'\cup a''\\b=b'\cup b''}}\epsilon(a',a'')\epsilon(b',b'')A_{a',b'}B_{a'',b''}.
\end{aligned}
\]
This multiplication is bilinear, commutative and associative whose properties are recorded in the following lemma.

\begin{lem}
For $A,B\in M_n(\C)$, $C\in M_{\left(\substack{n\\p}\right)}(\C)$, $D\in M_{\left(\substack{n\\q}\right)}(\C)$, we have
\begin{enumerate}
\item $A^{[p]}=A\sqcap...\sqcap A$ with $p$ copies,
\item $(A+B)^{[p]}=\sum_{a+b=p}\left(\substack{p\\a}\right)A^{[a]}\sqcap B^{[b]}$,
\item $A^{[p+q]}(C\sqcap D)=(A^{[p]}C)\sqcap (A^{[q]}D)$,
\item $(C\sqcap D)A^{[p+q]}=(CA^{[p]})\sqcap(DA^{[q]})$,
\item $(\mathrm{Ad}^{[p]}A)A^{[p]}=\det(A)1_{\left(\substack{n\\p}\right)}$.
\end{enumerate}
\end{lem}

We remark that for our purpose we actually hope to define $A^{[p]}$, $\mathrm{Ad}^{[p]}A$ and $\sqcap$-multiplication for matrices with quaternion entries. The above definitions with $\C$ replaced by other ring still make sense but the properties in above lemma no longer hold in general for non-commutative rings. Here we will not try to study these multi-linear algebra in the non-commutative setting. Instead, we will embed $\mathbb{H}$ into $M_2(\C)$ and consider all matrices with entries in $\mathbb{H}$ as complex matrices.

Write $\mathbb{H}=\R\oplus\R\mathbf{i}\oplus\R\mathbf{j}\oplus\R\mathbf{ij}$ with $\mathbf{i}^2=\mathbf{j}^2=-1$, $\mathbf{ij}=-\mathbf{ji}$ for the Hamilton quaternion algebra. Define an embedding
\[
\mathfrak{i}:\mathbb{H}\to M_2(\C),\qquad x=a+b\mathbf{i}+c\mathbf{j}+d\mathbf{ij}\mapsto\left[\begin{array}{cc}
a+bi & c+di\\
-c+di & a-bi
\end{array}\right].
\]
Extending this map entries by entries provides an embedding
\[
\mathfrak{i}:M_n(\mathbb{H})\to M_{2n}(\C),\qquad x=(x_{ij})\mapsto(\mathfrak{i}(x_{ij})).
\]
One checks that $\mathfrak{i}$ induces an isomorphism
\[
\mathfrak{i}:M_n(\mathbb{H})\stackrel{\sim}\longrightarrow M_{2n}'(\C):=\{x\in M_{2n}(\C):J_n'x=\overline{x}J_{n}'\},
\]
where $J_n'=\mathrm{diag}[J_1,...,J_1]$ with $n$ copies of $J_1=\left[\begin{array}{cc}
0 & -1\\
1 & 0
\end{array}\right]$. We thus have identifications
\[
\begin{aligned}
G_n(\R)\stackrel{\sim}\longrightarrow G_{2n}'(\R)&:=\{g\in\mathrm{SL}_{4n}(\C):gJ_{2n}g^{\ast}=J_{2n},J_{2n}'g=\overline{g}J_{2n}'\},\\
S_n(\R)\stackrel{\sim}\longrightarrow S_{2n}'(\R)&:=\{x\in M_{2n}'(\C):x^{\ast}=x\},\\
\mathcal{H}_n\stackrel{\sim}\longrightarrow\mathfrak{H}_{2n}&:=\{z=x+iy:x,y\in M_{2n}'(\C),x^{\ast}=x,y^{\ast}=y>0\}\\
&:=\{z\in M_{2n}(\C):J_n'z=\transpose{z}J_n,i(z^{\ast}-z)>0\}.
\end{aligned}
\]
Here the last map is given by $z=x+iy\mapsto\mathfrak{i}(x)+i\mathfrak{i}(y)$ and we set $z^{\ast}:=J_n'\transpose{z}J_n'=x^{\ast}+iy^{\ast}$ for $z=x+iy\in M_{2n}'(\C)\otimes_{\R}\C$. The group $G_{2n}'(\R)$ acts on $\mathfrak{H}_{2n}$ by
\[
g=\left[\begin{array}{cc}
a & b\\
c & d
\end{array}\right]\times z\mapsto(az+b)(cz+d)^{-1}
\]
which is compatible with the above identifications and the action of $G_n(\R)$ on $\mathcal{H}_n$. The doubling embedding $G_{2n}'(\R)\times G_{2n}'(\R)\to G_{4n}'(\R)$ and $\mathfrak{H}_{2n}\times\mathfrak{H}_{2n}\to\mathfrak{H}_{4n}$ defined as in \eqref{doublinggroup} and \eqref{doublingdomain} are also compatible with the above identifications. We will construct differential operators on $\mathfrak{H}_{4n}$ have similar transformation properties as \eqref{33} with respect to this doubling embedding. 

Over $\mathfrak{H}_{4n}$ there are basic differential operators $\partial_{ij}=\frac{\partial}{\partial\mathfrak{Z}_{ij}}$. We write elements $\mathfrak{Z}\in\mathcal{H}_{2n}$ and these differential operators into matrices
\[
\mathfrak{Z}=(\mathfrak{Z}_{ij})=\left[\begin{array}{cc}
z_1 & z_2\\
z_3 & z_4
\end{array}\right],\qquad\partial=(\partial_{ij})=\left[\begin{array}{cc}
\partial_1 & \partial_2\\
\partial_3 & \partial_4
\end{array}\right],
\]
with relations
\[
J_n'z_1=\transpose{z}_1J_n',\qquad J_n'z_4=\transpose{z}_4J_n',\qquad J_nz_3=\transpose{z}_2J_n,
\]
\[
J_n'\partial_1=\transpose{\partial}_1J_n',\qquad J_n'\partial_4=\transpose{\partial}_4J_n',\qquad J_n\partial_3=\transpose{\partial}_2J_n
\]

Let
\[
f_{\mathfrak{T}}(\mathfrak{Z})=e^{\lambda(\mathfrak{TZ})}=e^{\lambda(\tau_1z_1+\tau_3z_2+\tau_2z_3+\tau_4z_4)},\qquad\mathfrak{T}=\left[\begin{array}{cc}
\tau_1 & \tau_2\\
\tau_3 & \tau_4
\end{array}\right]\in S_{4n}'(\R).
\]
For $0\leq p\leq 2n$, we have 
\[
\partial_1^{[p]}f_{\mathfrak{T}}(\mathfrak{Z})=\transpose\tau_1^{[p]}f_{\mathfrak{T}}(\mathfrak{Z}),\,\partial_4^{[p]}f_{\mathfrak{T}}(\mathfrak{Z})=\transpose\tau_4^{[p]}f_{\mathfrak{T}}(\mathfrak{Z}),
\]
\[
\partial_2^{[p]}f_{\mathfrak{T}}(\mathfrak{Z})=\transpose{\tau}_3^{[p]}f_{\mathfrak{T}}(\mathfrak{Z}),\,\partial_3^{[p]}f_{\mathfrak{T}}(\mathfrak{Z})=\transpose{\tau}_2^{[p]}f_{\mathfrak{T}}(\mathfrak{Z}).
\]

\subsection{Construction of differential operators}

We now start the construction of differential operators. The construction is similar to \cite{B85} for symplectic groups with $\transpose{\mathfrak{Z}}=\mathfrak{Z}$ replaced by $J_{2n}'\mathfrak{Z}=\transpose{\mathfrak{Z}}J_{2n}'$.

For nonnegative integers $a,b,p$ with $a+b+p=2n$, put
\begin{equation}
\begin{aligned}
\delta(p,a,b)=z_2^{[a]}\transpose{\partial_4}^{[a]}\sqcap((1_n^{[p]}\sqcap z_2^{[b]}\transpose{\partial}_2^{[b]})(\mathrm{Ad}^{[b+p]}\transpose{\partial}_1)\transpose{\partial_3}^{[b+p]}).
\end{aligned}
\end{equation}

For nonnegative integer $p,q$ such that $p+q=2n$ put
\[
\Delta(p,q)=\sum_{a+b=q}(-1)^b\left(\begin{array}{c}q\\b\end{array}\right)\delta(p,a,b).
\]
For $f$ of the form
\[
f(\mathfrak{Z})=g(z_2,z_3,z_4)e^{\lambda(\tau z_1)},\qquad\tau\in S_{2n}'(\R),\,\det(\tau)\neq 0,
\]
the straightforward computations show that
\[
\begin{aligned}
\delta(p,a,b)f&=\left(\tau^{[a+b]}z_2^{[a+b]}\transpose{(\partial_4^{[a]}\sqcap\partial_3^{[b]}\transpose{\tau}^{-[b]}\partial_2^{[b]})\sqcap\transpose{\partial_3}^{[p]}}\right)f.\\
\Delta(p,q)f&=\left(\transpose{\partial_3}^{[p]}\sqcap\tau^{[q]}z_2^{[q]}\transpose{(\partial_4-\partial_3\transpose{\tau}^{-1}\partial_2)^{[q]}}\right)f.
\end{aligned}
\]

\begin{lem}
\label{lem4.1}
Let $f$ be a holomorphic function on $\mathfrak{H}_{4n}$. Let $V$ be the operator defined by
\[
(f|V)\left(\left[\begin{array}{cc}
z_1 & z_2\\
z_3 & z_4
\end{array}\right]\right)=f\left(\left[\begin{array}{cc}
z_4 & z_3\\
z_2 & z_1
\end{array}\right]\right),
\]
then
\begin{equation}
\Delta(p,q)(f|V)=(\Delta(p,q)f)|V.
\end{equation}
\end{lem}

\begin{proof}
This is an analogue of the computations in \cite[page. 87]{B85}. It suffices to prove for certain test functions
\[
f_{\mathfrak{T}}(\mathfrak{Z})=e^{\lambda(\mathfrak{T}\mathfrak{Z})}=e^{\lambda(\tau_1z_1+\tau_2z_3+\tau_3z_2+\tau_4z_4)},\,\mathfrak{T}=\left[\begin{array}{cc}
\tau_1 & \tau_2\\
\tau_3 & \tau_4
\end{array}\right],
\]
with $\det(\tau_1)\neq 0$. We calculate that
\[
(\Delta(p,q)f)|V=\left(\tau_2^{[p]}\sqcap\tau_1^{[q]}z_3^{[q]}(\tau_4-\tau_3\tau_1^{-1}\tau_2)^{[q]}\right)f.
\]
On the other hand, $(f_{\mathfrak{T}}|V)(\mathfrak{Z})=e^{\lambda(\tau_1z_4+\tau_2z_2+\tau_3z_3+\tau_4z_4)}$, so
\[
\Delta(p,q)(f|V)=\left(\tau_3^{[p]}\sqcap\tau_4^{[q]}z_2^{[q]}(\tau_1-\tau_2\tau_4^{-1}\tau_3)^{[q]}\right)f.
\]
The lemma then follows from properties of $\sqcap$ and the fact $A^{[p]}\sqcap B^{[q]}=\transpose A^{[p]}\sqcap\transpose B^{[q]}$ if $p+q=2n$.
\end{proof}

Keep considering $f:=f_{\mathfrak{T}}$ as a test function with $\det(\tau_1)\neq 0$. We are now going to compute $\Delta(p,q)(f|_{\kappa}1\times J_{2n})$ and prove the analogue of \cite[Proposition 1]{B85}. Clearly,
\[
(1\times J_{2n})\mathfrak{Z}=\left[\begin{array}{cc}
z_1-z_2z_4^{-1}z_3 & z_2z_4^{-1}\\
z_4^{-1}z_3 & -z_4^{-1}
\end{array}\right],
\]
\[
(f|_{\kappa}1\times J_{2n})(\mathfrak{Z})=\det(z_4)^{-\kappa}e^{\lambda(\tau_1(z_1-z_2z_4^{-1}z_3)+\tau_2z_4^{-1}z_3+\tau_3z_2z_4^{-1}-\tau_4z_4^{-1})}.
\]
Decompose $(f|_{\kappa}1\times J_{2n})(\mathfrak{Z})=g(z_2,z_3,z_4)h(z_4)e^{\lambda(\tau_1z_1)}$ with
\[
\begin{aligned}
g(z_2,z_3,z_4)&=\det(z_4)^{-2n}e^{\lambda(-\tau_1(z_2-\tau_1^{-1}\tau_2)z_4^{-1}(z_2-\tau_1^{-1}\tau_2)^{\ast}}),\\
h(z_4)&=\det(z_4)^{2n-\kappa}e^{\lambda((\tau_3\tau_1^{-1}\tau_2-\tau_4)z_4^{-1})}.
\end{aligned}
\]

\begin{lem}
\label{lemma4.3}
Let $\tau=\tau^{\ast}$ with $\det(\tau)\neq 0$. Then for $q\geq 1$, we have
\[
(\partial_4+\partial_3\transpose{\tau}^{-1}\partial_2)^{[q]}\det(z_4)^{-2n}e^{\lambda(\tau z_2z_4^{-1}z_2^{\ast})}=0.
\] 
\end{lem}

\begin{proof}
This is an analogue of \cite[Lemma 4]{B85}. As there we prove for $\mathfrak{Z}=\mathfrak{X}+i\mathfrak{Y}$ replaced by $\mathfrak{Y}$. We start with
\[
\pi^{-2n^2}\int_{M_{2n}'(\C)}e^{-\lambda(xy_4x^{\ast})}dx=\det(y_4)^{-2n}.
\]
Changing variables $x\mapsto uy_2y_4^{-1}$ with $\tau=u^{\ast}u$ we have
\[
\det(y_4)^{-2n}=\pi^{-2n^2}\int_{M_{2n}'(\C)}e^{-\lambda(xy_4x^{\ast}+xy^{\ast}_2u^{\ast}+uy_2x^{\ast})}dx\cdot e^{-\lambda(\tau(y_2y_4^{-1}y_2^{\ast}))}.
\]
Applying $(\partial_4+\partial_3\transpose{\tau}^{-1}\partial_2)^{[q]}$ we obtain
\[
\begin{aligned}
&(\partial_4+\partial_3\transpose{\tau}^{-1}\partial_2)^{[q]}\det(z_4)^{-2n}e^{\lambda(\tau z_2z_4^{-1}z_2^{\ast})}\\
=&\pi^{-2n^2}\int_{M_{2n}'(\C)}(\partial_4+\partial_3\transpose{\tau}^{-1}\partial_2)^{[q]}e^{-\lambda(xy_4x^{\ast}+xy^{\ast}_2u^{\ast}+uy_2x^{\ast})}dx\\
=&\pi^{-2n^2}\int_{M_{2n}'(\C)}\sum_{a+b=q}\left(\begin{array}{c}q\\a\end{array}\right)(-\transpose{x}\overline{x})^{[a]}\sqcap(\transpose{x}\overline{x})^{[b]}e^{-\lambda(xy_4x^{\ast}+xy^{\ast}_2u^{\ast}+uy_2x^{\ast})}dx=0
\end{aligned}
\]
unless $q=0$ which proves the claim.
\end{proof}

\begin{lem}
\label{lemma4.4}
Let $h$ be a holomorphic function on $\mathfrak{H}_{2n}$ (in variable $z_4$), then
\[
\begin{aligned}
\partial_4^{[q]}(h(-z_4^{-1})\det(z_4)^{-\kappa})&=\det(z_4)^{-\kappa}\sum_{a+b=q}(-1)^a\left(\begin{array}{c}q\\a\end{array}\right)C_a\left(\kappa-\frac{q-1}{2}\right)\transpose{z_4^{-[q]}}\\
&\times (\transpose{z}_4^{-[a]}\sqcap(\partial^{[b]}h)(-z_4^{-1}))\transpose{z}_4^{-[q]},
\end{aligned}
\]
where
\[
C_a(s):=s\left(s+\frac{1}{2}\right)...\left(s+\frac{a-1}{2}\right).
\]
\end{lem}

\begin{proof}
This is an analogue of \cite[Lemma 5]{B85} and can be proved by repeating the same computations in \cite[Section 3.1]{Bo}. For completeness, we include the proof of the lemma in the Appendix \ref{appendix}.
\end{proof}

The same proof of \cite[Lemma 6]{B85} shows that
\[
(\partial_4-\partial_3\transpose\tau_1^{-1}\partial_2)^{[q]}(gh)=\sum_{a+b=q}\left(\begin{array}{c}
q\\
a
\end{array}\right)((\partial_1-\partial_3\transpose{\tau}_1^{-1}\partial_2)^{[a]}g)\sqcap(\partial_4^{[b]}h).
\]
By Lemma \ref{lemma4.3} the $(\partial_1-\partial_3\transpose{\tau}_1^{-1}\partial_2)^{[a]}g=0$ unless $a=0$ and thus by Lemma \ref{lemma4.4} we have
\[
\begin{aligned}
(\partial_4-\partial_3\transpose{\tau}_1^{-1}\partial_2)^{[q]}f|_{\kappa}1\times J_n&=(-1)^q\det(z_4)^{2n-\kappa}\sum_{a+b=q}\left(\begin{array}{c}
q\\a
\end{array}\right)C_a\left(\kappa-2n-\frac{q-1}{2}\right)\transpose{z}_4^{-[q]}\\
&\times(\transpose{z}_4^{-[a]}\sqcap\transpose(\tau_3\tau_1^{-1}\tau_2-\tau_4)^{[b]})\transpose{z}_4^{-[q]}e^{\lambda(\mathfrak{T}(1\times J_n)\mathfrak{Z})}.
\end{aligned}
\]

Note that
\[
\begin{aligned}
\transpose\partial_3^{[p]}g&=(-1)^p(\tau_1z_2-\tau_2)^{[p]}z_4^{-[p]}g.
\end{aligned}
\]
The following proposition can then be obtained by straightforward computations.

\begin{prop}
Let $f=f_{\mathfrak{T}}$ be a test function with $\det(\tau_1)\neq0$ then
\begin{equation}
\begin{aligned}
&\Delta(p,q)(f|_{\kappa}1\times J_{2n})=\sum_{a+b=q}(-1)^a\left(\begin{array}{c}q\\a\end{array}\right)C_a\left(\kappa-2n-\frac{q-1}{2}\right)\\
\times&\left((\tau_1z_2-\tau_2)^{[p]}\sqcap\tau_1^{[q]}z_2^{[q]}(1^{[a]}\sqcap z_4^{-[b]}(\tau_3^{\ast}\tau_1^{-1}\tau_2-\tau_4)^{[b]})\right)f|_{{\kappa}+1}1\times J_{2n}.
\end{aligned}
\end{equation}
\end{prop}

Setting
\[
\mathfrak{D}_{\kappa}=\sum_{p+q=2n}\lambda_{\kappa}(q)\Delta(p,q),
\]
we will determine the coefficient $\lambda_{\kappa}(q)$ such that
\[
\mathfrak{D}_{\kappa}(f|_{\kappa}1\times J_{2n})=(\mathfrak{D}_{\kappa}f)|_{\kappa+1}1\times J_{2n}.
\]
Let $f=f_{\tau}$ be the test function with $\det(\tau_1)\neq0$ as previous. By above proposition, we can write
\[
\begin{aligned}
\mathfrak{D}_{\kappa}(f|_{\kappa}1\times J_{2n})&=\sum_{p+q=2n}\sum_{c+d=p}\sum_{a+b=q}(-1)^{a+b+d}\lambda_{\kappa}(q)\left(\begin{array}{c}p\\c\end{array}\right)\left(\begin{array}{c}q\\a\end{array}\right)C_a\left(-\kappa+2n+\frac{b}{2}\right)\\
&\times\left((\tau_1z_2)^{[c+q]}(1^{[a+c]}\sqcap z_4^{-[b]}(\tau_4-\tau_3\tau_1^{-1}\tau_2)^{[b]})\sqcap\tau_2^{[d]}\right)f|_{\kappa+1}1\times J_{2n}.
\end{aligned}
\]
Denote $t=a+c$ and put
\[
a_{\kappa}(t,b,d)=\sum_{a+c=t}(-1)^{a+b+d}\lambda_{\kappa}(a+b)\left(\begin{array}{c}
a+b\\a
\end{array}\right)\left(\begin{array}{c}
c+d\\c
\end{array}\right)C_a\left(-\kappa+2n+\frac{b}{2}\right).
\]
Then we have
\[
\begin{aligned}
&\mathfrak{D}_{\kappa}(f|_{\kappa}1\times J_{2n})\\
=&\sum_{t+b+d=2n}a_{\kappa}(t,b,d)\left((\tau_1z_2)^{[t+b]}(1^{[t]}\sqcap z_4^{-[b]}(\tau_4-\tau_3\tau_1^{-1}\tau_2)^{[b]})\sqcap\tau_2^{[d]}\right)f|_{{\kappa}+1}1\times J_{2n}.
\end{aligned}
\]
One computes that
\[
\begin{aligned}
(\mathfrak{D}_{\kappa}f)|_{{\kappa}+1}1\times J_{2n}&=\sum_{p+q=2n}\lambda_{\kappa}(q)\left((\tau_1z_2z_4^{-1})^{[q]}(\tau_4-\tau_3\tau_1^{-1}\tau_2)^{[q]}\sqcap\tau_2^{[p]}\right)f|_{{\kappa}+1}1\times J_{2n}.
\end{aligned}
\]
Therefore, in order to have the desired property, we must have $a_{\kappa}(0,q,p)=\lambda_{\kappa}(q)$ and $a_{\kappa}(t,b,d)=0$ otherwise. By the same argument as in \cite[page. 89]{B85}, 
\[
\lambda_{\kappa}(q)=\left(\begin{array}{c}
2n\\q
\end{array}\right)\tilde{C}_q(-{\kappa}+2n),\qquad\tilde{C}_q(s)=\prod_{\substack{0\leq p\leq 2n\\p\neq q}}C_p(s),
\]
has the desired property. To make the formula simpler, we normalize $\lambda_{\kappa}(q)$ by multiplying $\frac{C_{2n}(-\kappa+2n)}{\widetilde{C}_0(-\kappa+2n)}$ and define
\begin{equation}
\mathfrak{D}_{\kappa}=\sum_{p+q=2n}(-1)^p\left(\begin{array}{c}2n\\p\end{array}\right)C_q\left(\kappa-3n+\frac{1}{2}\right)\Delta(p,q),
\end{equation}
which satisfies the transformation properties as \eqref{33}.

\subsection{Differential operators}

We now treat Case I and II together. For any $\alpha\in\C$, we set\\
(Case I)
\[
\mathfrak{D}_{\alpha}=\sum_{p+q=2n}(-1)^p\left(\begin{array}{c}2n\\p\end{array}\right)C_q\left(\alpha-3n+\frac{1}{2}\right)\Delta(p,q),
\]
(Case II)
\[
\mathfrak{D}_{\alpha}=\sum_{p+q=2n}(-1)^p\left(\begin{array}{c}2n\\p\end{array}\right)C_q\left(\alpha-2n+\frac{1}{2}\right)\Delta(p,q).
\]
and summarize our construction of differential operators in the following proposition.

\begin{prop}
There is a differential operator $\mathfrak{D}_{\alpha}$ on $\mathcal{H}_{2n}$ with $\alpha\in\C$ such that for any integer $\kappa>2n$ and any holomorphic function $f$ and $\gamma\in G(\R)$ we have
\begin{equation}
\begin{aligned}
\mathfrak{D}_{\kappa}(f|V)&=(\mathfrak{D}_{\kappa}f)|V,\\
\mathfrak{D}_{\kappa}(f|_{\kappa}1\times\gamma)&=(\mathfrak{D}_{\kappa}f)|_{\kappa+1}1\times\gamma,\\
\mathfrak{D}_{\kappa}(f|_{\kappa}\gamma\times 1)&=(\mathfrak{D}_{\kappa}f)|_{\kappa+1}\gamma\times1.
\end{aligned}
\end{equation}
\end{prop}

\begin{proof}
It suffices to prove these formulas for test functions  and generators of $G(\R)$. The first and second formula follows easily from our previous discussions and the third formula follows from the first and second. 
\end{proof}

For $s\in\C$ and an integer $v\geq 1$ we define
\[
\mathfrak{D}_{s}^v=\mathfrak{D}_{s+v-1}\circ...\circ\mathfrak{D}_s,\qquad\mathring{\mathfrak{D}}_{s}^v=(\mathfrak{D}_{s}^v)|_{z_2=0}.
\]
For $\mathfrak{T}=-\epsilon\mathfrak{T}^{\ast}$ with $\epsilon=-1$ in Case I and $\epsilon=1$ in Case II, we define a polynomial $\mathfrak{P}_{s}^v$ in the entries $\mathfrak{T}_{ij}$ of $\mathfrak{T}$ by
\[
\mathring{\mathfrak{D}_{s}^v}(e^{\lambda(\mathfrak{T}\mathfrak{Z})})=\mathfrak{P}_{s}^v(\mathfrak{T})e^{\lambda(\tau_1z_1+\tau_4z_4)},\qquad\mathfrak{T}=\left[\begin{array}{cc}
\tau_1 & \tau_2\\
\tau_3 & \tau_4
\end{array}\right].
\]
$\mathring{\mathfrak{D}}_{s}^v$ is a homogeneous polynomial in $\partial_{ij}$ with at most one term free of $\partial_1$ and $\partial_4$, namely $c_{s}^v\det(\partial_2)^{v}$ with certain constant $c_{s}^v$. To determine the constant, we observe that\\
(Case I)
\[
\mathfrak{D}_{s}(\det(z_2)^v)=C_{2n}\left(\frac{v}{2}\right)C_{2n}\left(s-3n+\frac{v}{2}\right)\det(z_2)^{v-1},
\]
(Case II)
\[
\mathfrak{D}_{s}(\det(z_2)^v)=C_{2n}\left(\frac{v}{2}\right)C_{2n}\left(s-2n+\frac{v}{2}\right)\det(z_2)^{v-1}.
\]
Hence\\
(Case I)
\[
\mathring{\mathfrak{D}}_{s}^v(\det(z_2)^v)=\prod_{i=1}^{v}C_{2n}\left(\frac{i}{2}\right)C_{2n}(s-3n+v-\frac{i}{2}),
\]
(Case II)
\[
\mathring{\mathfrak{D}}_{s}^v(\det(z_2)^v)=\prod_{i=1}^{v}C_{2n}\left(\frac{i}{2}\right)C_{2n}(s-2n+v-\frac{i}{2}).
\]
Therefore
\begin{equation}
\label{csv}
\begin{aligned}
c_s^v&=\prod_{i=0}^{v-1}C_{2n}(s-3n+v-\frac{i}{2}),\qquad&\text{ Case I},\\ 
c_s^v&=\prod_{i=0}^{v-1}C_{2n}(s-2n+v-\frac{i}{2}),\qquad&\text{ Case II}.
\end{aligned}
\end{equation}
Let $p$ be an odd prime. Then for $\mathfrak{T}=\left[\begin{array}{cc}
\tau_1 & \tau_2\\
\tau_3 & \tau_4
\end{array}\right]\in\Lambda$ such that $2L^{-1}\tau_1,2L^{-1}\tau_4\in M_n(\mathcal{O}^{\#})$ with $L=p^n$ a power of $p$, there is a congruence
\begin{equation}
\mathfrak{P}_s^v(\mathfrak{T})\equiv c_s^v\det(\tau_2)^v\text{ mod }L.
\label{41}
\end{equation}

\subsection{Classical reformulations}

Let $\mathbf{f}\in S_{\kappa}(K_0(\mathfrak{n}p))$ be an eigenform as in Section \ref{section 3}. Denote $f\in S_{\kappa}(\Gamma_0(Np))$ be its corresponding classical modular forms. Let $\chi\neq 1$ be a primitive Dirichlet character of conductor $p^{\mathfrak{c}}$. Define 
\begin{equation}
\mathfrak{E}_{\kappa}(z,w,s):=j(g_z,z_0)^{\kappa}j(h_z,z_0)^{\kappa}\mathcal{E}_{\kappa}(g_z,h_z,s),
\end{equation}
for $z=g_zz_0,w=h_zz_0$ be a function on $\mathcal{H}\times\mathcal{H}$. Here $\mathcal{E}_{\kappa}$ is defined using the Hecke character corresponds to $\chi$. Clearly $\mathfrak{E}_{\kappa}(z,w,\kappa)\in M_{\kappa}(\Gamma_0(N^2p^{2\mathfrak{c}}))\otimes M_{\kappa}(\Gamma_0(N^2p^{2\mathfrak{c}}))$. Similarly, for $\chi=1$ define
\begin{equation}
\mathfrak{E}_{\kappa}(z,w,s,i,r):=j(g_z,z_0)^{\kappa}j(h_z,z_0)^{\kappa}\mathcal{E}_{\kappa}(g_z,h_z,s,i,r).
\end{equation}
Then $\mathfrak{E}_{\kappa}(z,w,\kappa,i,r)\in M_{\kappa}(\Gamma_0(N^2p^2))\otimes M_{\kappa}(\Gamma_0(N^2p^2))$. The inner product formula in Proposition \ref{prop 3.6} and Proposition \ref{prop 3.8} can be reformulated as
\begin{equation}
\label{44}
\left\langle \mathfrak{E}_{\kappa}(z,-\overline{w},s),f\left|_{\kappa}\left[\begin{array}{cc}
0 & \epsilon\\
N^2p^{2\mathfrak{c}} & 0
\end{array}\right](z)\right.\right\rangle_{\Gamma_0(N^2p^{2\mathfrak{c}})}^z=\frac{\Omega^{\mathfrak{c}}_{\kappa}(s)}{\Lambda^{2n}_{Np}(s,\chi)}L_{Np}(s,f,\chi)\overline{f(w)},
\end{equation}
for $\chi\neq 1$ and
\begin{equation}
\label{45}
\begin{aligned}
&\left\langle\sum_{\substack{0\leq i\leq n\\r=0,1}}(-1)^rp^{\frac{(2i+r)(2i+r-1)}{2}-2n(2i+r)}\mathfrak{E}_{\kappa}(z,-\overline{w},s,i,r),f\left|_{\kappa}\left[\begin{array}{cc}
0 & \epsilon\\
N^2p^{2} & 0
\end{array}\right](z)\right.\right\rangle^z_{\Gamma_0(Np)}\\
=&\frac{p^{-2n^2-2n+2ns}\Omega^0_{\kappa}(s)}{\Lambda^{2n}_{Np}(s)}\prod_{i=1}^{2n}(1-\beta_{i,p}p^{-s+2n+\epsilon})L_N(s,f)\overline{f(w)}.
\end{aligned}
\end{equation}

Let $v\geq 0$. For $\chi\neq 1$, define
\begin{equation}
\mathfrak{E}_{\kappa}^v(z,w,s):=\delta(z)^{s-\kappa+v}\delta(w)^{s-\kappa+v}\mathfrak{D}_s^v\left(\delta(z)^{-s+\kappa-v}\delta(w)^{-s+\kappa-v}\mathfrak{E}_{\kappa-v}(z,w,s)\right).
\end{equation}
For $\chi=1$ define
\begin{equation}
\mathfrak{E}_{\kappa}^v(z,w,s,i,r):=\delta(z)^{s-\kappa+v}\delta(w)^{s-\kappa+v}\mathfrak{D}_s^v\left(\delta(z)^{-s+\kappa-v}\delta(w)^{-s+\kappa-v}\mathfrak{E}_{\kappa-v}(z,w,s,i,r)\right).
\end{equation}

\begin{prop}
Assume $\kappa>2n$ and $f\in S_{\kappa}(\Gamma_0(Np))$ is an eigenform. Denote
\begin{equation}
\label{omega}
\begin{aligned}
\Omega^{\mathfrak{c}}_{\kappa,v}(s)&=\prod_{i=0}^{v-1}C_{2n}\left(s+i-3n+\frac{1}{2}\right)C_{2n}(-s-i)\Omega^{\mathfrak{c}}_{\kappa}(s),\qquad &\text{Case I},\\
\Omega^{\mathfrak{c}}_{\kappa,v}(s)&=\prod_{i=0}^{v-1}C_{2n}\left(s+i-2n+\frac{1}{2}\right)C_{2n}(-s-i)\Omega^{\mathfrak{c}}_{\kappa}(s),\qquad &\text{Case II}.
\end{aligned}
\end{equation}
(1) Let $\chi$ be a nontrivial Hecke character of conductor $p^{\mathfrak{c}}$, then
\begin{equation}
\left\langle \mathfrak{E}^v_{\kappa}(z,-\overline{w},s),f\left|_{\kappa}\left[\begin{array}{cc}
0 & \epsilon\\
N^2p^{2\mathfrak{c}} & 0
\end{array}\right](z)\right.\right\rangle_{\Gamma_0(N^2p^{2\mathfrak{c}})}^z=\frac{\Omega^{\mathfrak{c}}_{\kappa,v}(s)}{\Lambda^{2n}_{Np}(s,\chi)}L_{Np}(s,f,\chi)\overline{f(w)}.\label{48}
\end{equation}
(2) Let $\chi=1$, then
\begin{equation}
\begin{aligned}
&\left\langle\sum_{\substack{0\leq i\leq n\\r=0,1}}(-1)^rp^{\frac{(2i+r)(2i+r-1)}{2}-2n(2i+r)}\mathfrak{E}^v_{\kappa}(z,-\overline{w},s,i,r),f\left|_{\kappa}\left[\begin{array}{cc}
0 & \epsilon\\
N^2p^{2} & 0
\end{array}\right](z)\right.\right\rangle^z_{\Gamma_0(Np)}\\
=&\frac{p^{-2n^2-2n+2ns}\Omega^0_{\kappa,v}(s)}{\Lambda^{2n}_{Np}(s)}\prod_{i=1}^{2n}(1-\beta_{i,p}p^{-s+2n+\epsilon})L_{N}(s,f)\overline{f(w)}.
\end{aligned}
\end{equation}
\label{prop4.5}
\end{prop}

\begin{proof}
Similar to the computations in \cite[Section 4]{B85} we need to apply the differential operators to 
\[
f_s(\mathfrak{Z})=\det(z_1+z_2+z_3+z_4)^{-s},\quad\mathfrak{Z}=\left[\begin{array}{cc}
z_1 & z_2\\
z_3 & z_4
\end{array}\right]
\]
which is\\
(Case I)
\[
\mathfrak{D}_s^vf_s=\prod_{i=0}^{v-1}C_{2n}(s+i-3n+\frac{1}{2})C_{2n}(-s-i)f_{s+v},
\]
(Case II)
\[
\mathfrak{D}_s^vf_s=\prod_{i=0}^{v-1}C_{2n}(s+i-2n+\frac{1}{2})C_{2n}(-s-i)f_{s+v}.
\]
Hence, up to those constants
\[
\mathfrak{D}_s^v(j(g,z)^{-\kappa+v}|j(g,z)|^{\kappa-v-s})=j(g,z)^{-\kappa}|j(g,z)|^{\kappa-s}
\]
and the proposition can be easily obtained from \eqref{44},\eqref{45}.
\end{proof}

\section{The Main Analytic Result}
\label{section 5}

Let $p\nmid N$ be an odd prime split in $\mathbb{B}$ as in previous sections. Let $f\in S_{\kappa}(\Gamma_0(Np))$ be an eigenform for all Hecke operators and also an eigenform of $U(L)$ for any $L|p^{\infty}$, i.e $f|U(L)=\alpha(L)f$. Suppose $\chi\neq 1$ is a nontrivial primitive Dirichlet character of conductor $p^{\mathfrak{c}}$. We rewrite the inner product formula in \eqref{48} as
\begin{equation}
\begin{aligned}
&\left\langle \mathfrak{E}^v_{\kappa}(z,-\overline{w},s)\left|^z_{\kappa}\left[\begin{array}{cc}
1 & 0\\
0 & N^2p^{2\mathfrak{c}}
\end{array}\right]\right.\left|^w_{\kappa}\left[\begin{array}{cc}
1 & 0\\
0 & N^2p^{2\mathfrak{c}}
\end{array}\right]\right.,f\left|_{\kappa}\left[\begin{array}{cc}
0 & \epsilon\\
1 & 0
\end{array}\right](z)\right.\right\rangle_{\Gamma^0(N^2p^{2\mathfrak{c}})}^z\\
=&\frac{\Omega^{\mathfrak{c}}_{\kappa,v}(s)}{\Lambda^{2n}_{Np}(s,\chi)}L_{Np}(s,f,\chi)\overline{f(w)}.
\end{aligned}
\end{equation}
This is a Petersson inner product of level $N^2p^{2\mathfrak{c}}$. We are going to express it as a Petersson inner product of level $N^2p$. For $f,g\in M_{\kappa}(\Gamma^0(N^2p^{2\mathfrak{c}}))$, there is a trace operator
\[
\mathrm{Tr}f:=\sum_{\gamma}f|_{\kappa}\gamma,\mathrm{Tr}g:=\sum_{\gamma}g|_{\kappa}\gamma,
\]
where $\gamma$ runs over $\Gamma^0(N^2p^{2\mathfrak{c}})\backslash \Gamma^0(N^2p)$ such that $\mathrm{Tr}f,\mathrm{Tr}g\in M_{\kappa}(\Gamma^0(N^2p))$. Moreover, we have
\[
\langle f,g\rangle_{\Gamma^0(N^2p^{2\mathfrak{c}})}=\left\langle\mathrm{Tr}f,\mathrm{Tr}g\right\rangle_{\Gamma^0(N^2p)}.
\]
In particular, the representatives can be chosen as
\[
\left\{\left[\begin{array}{cc}
1 & N^2pT\\
0 & 1
\end{array}\right]:T=\epsilon T^{\ast}\in M_n(\mathcal{O})\text{ mod }p^{2\mathfrak{c}-1}\right\}.
\]
By the fact 
\[
\left[\begin{array}{cc}
1 & 0\\
0 & N^2p^{2\mathfrak{c}}
\end{array}\right]\left[\begin{array}{cc}
1 & N^2pT\\
0 & 1
\end{array}\right]=\left[\begin{array}{cc}
1 & T\\
0 & p^{2\mathfrak{c}-1}
\end{array}\right]\left[\begin{array}{cc}
1 & 0\\
0 & Np
\end{array}\right],
\]
we have
\[
\mathrm{Tr}\left(f\left|_{\kappa}\left[\begin{array}{cc}
1 & 0\\
0 & N^2p^{2\mathfrak{c}}
\end{array}\right]\right.\right)=f|U(p^{2\mathfrak{c}-1})\left|_{\kappa}\left[\begin{array}{cc}
1 & 0\\
0 & N^2p
\end{array}\right]\right..
\]
For any $p$-power integer $L$ with $p^{\mathfrak{c}}|L$, define
\begin{equation}
\begin{aligned}
&\mathfrak{g}(z,w):=\mathfrak{g}(z,w,\kappa,v,\chi,L)\\
=&\Lambda_{Np}^{2n}(\kappa-v,\chi)\mathfrak{E}_{\kappa}^v(z,w,\kappa-v)|^zU(L^2)|^wU(L^2)\left|_{\kappa}^z\left[\begin{array}{cc}
1 & 0\\
0 & N^2p
\end{array}\right]\right.\left|_{\kappa}^w\left[\begin{array}{cc}
1 & 0\\
0 & N^2p
\end{array}\right]\right.,
\end{aligned}
\label{51}
\end{equation}
and similarly for $\chi=1$ define 
\begin{equation}
\label{52}
\begin{aligned}
&\mathfrak{g}(z,w):=\mathfrak{g}(z,w,\kappa,v,1,L)\\
=&\sum_{i,r}(-1)^rp^{\frac{(2i+r)(2i+r-1)}{2}-2n(2i+r)}\Lambda_{Np}^{2n}(\kappa-v)\\
\times&\mathfrak{E}^{v}_{\kappa}(z,w,\kappa-v,i,r)|^zU(L^2)|^wU(L^2)\left|_{\kappa}^z\left[\begin{array}{cc}
1 & 0\\
0 & N^2p
\end{array}\right]\right.\left|_{\kappa}^w\left[\begin{array}{cc}
1 & 0\\
0 & N^2p
\end{array}\right]\right..
\end{aligned}
\end{equation}
Denote
\begin{equation}
\label{53}
E_p(s,\chi)=\prod_{i=1}^{2n}\frac{1-\beta_{i,p}^{-1}\chi(p)p^{s-2n-\epsilon}}{1-\beta_{i,p}\chi(p)p^{2n-s+\epsilon-1}},
\end{equation}
so $E_p(s,\chi)=1$ unless $\chi=1$. Then by above discussions and straightforward computations we have
\begin{equation}
\begin{aligned}
&\left\langle\left\langle\mathfrak{g}(z,-\overline{w}),f\left|_{\kappa}\left[\begin{array}{cc}
0 & \epsilon\\
1 & 0
\end{array}\right]\right.\right\rangle^z_{\Gamma^0(N^2p)},f\left|_{\kappa}\left[\begin{array}{cc}
1 & 0\\
0 & N^2p
\end{array}\right]\right.\right\rangle^w_{\Gamma^0(N^2p)}\\
=&\Omega^{\mathfrak{c}}_{\kappa,v}(\kappa-v)
\frac{\alpha(pL^4)}{\alpha(p^{2\mathfrak{c}})}\langle f,f\rangle_{\Gamma^0(Np)}L_{Np}(\kappa-v,f,\chi)E_p(\kappa-v,\chi).
\end{aligned}
\label{54}
\end{equation}
Here $\Omega_{\kappa,v}^{\mathfrak{c}}(\kappa-v)$ is the constant defined in \eqref{omega}.

We end up this section by recording the Fourier expansion of $\mathfrak{g}(z,w)$ in the following proposition.

\begin{prop}
\label{5.1}
Let $\kappa,v\in\Z$ with $\kappa-v>2n$ and $\chi$ a primitive Dirichlet character of conductor $p^{\mathfrak{c}}$. Denote $\mathfrak{g}(z,w)$ for the function defined by \eqref{51} or \eqref{52}. Then $\mathfrak{g}(z,w)$ has Fourier expansion\\
(Case I)
\begin{equation}
\begin{aligned}
\mathfrak{g}(z,w)&=p^{\mathfrak{c}n(n-1)}G(\chi)^n(N^2p)^{-2\kappa n}\sum_{\mathfrak{T}}\mathfrak{P}_{\kappa-v}^v(\mathfrak{T})\chi^{-1}(\det(2\tau_2))e(2N^{-1}\lambda(\tau_2))\\
&\times A_{\kappa-v}(2n)\det(\mathfrak{T})^{\kappa-v-\frac{4n-1}{2}}\prod_{l\nmid Np}P(\chi(l)l^{-\kappa+v})e(\lambda((N^2p)^{-1}(\tau_1z+\tau_4w))),
\end{aligned}
\end{equation} 
(Case II)
\begin{equation}
\begin{aligned}
\mathfrak{g}(z,w)&=p^{\mathfrak{c}n(n-1)}G(\chi)^n(N^2p)^{-2\kappa n}\sum_{\mathfrak{T}}\mathfrak{P}_{\kappa-v}^v(\mathfrak{T})\chi^{-1}(\det(2\tau_2))e(2N^{-1}\lambda(\tau_2))\\
&\times A_{\kappa-v}(2n)\det(\mathfrak{T})^{\kappa-v-\frac{4n+1}{2}}\prod_{l\nmid Np}P(\chi(l)l^{-\kappa+v})\\
&\times L_{Np}(\kappa-v-2n,\chi\rho_{\mathfrak{T}})e(\lambda((N^2p)^{-1}(\tau_1z+\tau_4w))).
\end{aligned}
\end{equation} 
Here $\mathfrak{T}=\left[\begin{array}{cc}
L^2\tau_1 & \tau_2\\
\tau_3 & L^2\tau_4
\end{array}\right]\in\Lambda_{2n}^+$ with $\tau_1,\tau_4\in\Lambda_n,2\tau_2\in M_n(\mathcal{O}^{\#})$, $p\nmid\det(\mathfrak{T})$,
\begin{equation}
\label{A(2n)}
\begin{aligned}
A_{\kappa-v}(2n)&=2^{2-4n}(2\pi i)^{4n(\kappa-v)}(4D_{\mathbb{B}})^{-n(2n-1)}\Gamma_{2n}(2\kappa-2v)^{-1},&\qquad\text{ Case I},\\
A_{\kappa-v}(2n)&=2^{1-\frac{4n+1}{2}}(2\pi i)^{4n(\kappa-v)}(4D_{\mathbb{B}})^{-n(2n+1)}\Gamma_{4n}(2\kappa-2v)^{-1},&\qquad\text{ Case II},
\end{aligned}
\end{equation}
is defined in \eqref{11}, \eqref{12} and
\begin{equation}
\label{57}
G(\chi)=\sum_{x\in\mathcal{O}_p/p^{\mathfrak{c}}\mathcal{O}_p}\chi(N(x))e(p^{-\mathfrak{c}}\lambda(x)).
\end{equation}

\end{prop}

\begin{proof}
Assume $\chi\neq 1$. For $\tau\in\Lambda_{2n}$, we write it as $\tau=\left[\begin{array}{cc}
\tau_1 & \tau_2\\
-\epsilon\tau_2^{\ast} & \tau_4
\end{array}\right]$ with $\tau_1,\tau_4\in\Lambda_n,\tau_2\in M_n(\mathcal{O})$. By our definition of $\mathfrak{E}_{\kappa}(z,w,\kappa)$, it has a Fourier expansion of the form
\[
\begin{aligned}
\Lambda_{Np}^{2n}(\kappa,\chi)\mathfrak{E}_{\kappa}(z,w,\kappa)&=\sum_{\tau\in\Lambda_{2n}}a(\tau,\mathrm{diag}[y,v],\kappa,\chi)e(\lambda(hz))\\
&\times\sum_{x\in M_n(\mathcal{O}_p)/p^{\mathfrak{c}}M_n(\mathcal{O}_p)}\chi(\det x)^{-1}e\left(\frac{2\lambda(\tau_2 x)}{p^{\mathfrak{c}}}\right)e\left(\frac{2\lambda(\tau_2)}{N}\right),
\end{aligned}
\]
where $z=x+iy,w=u+iv$ and $a(\tau,\mathrm{diag}[y,v],\kappa,\chi)$ are the Fourier coefficients of $E_{\kappa}(z\times w,\kappa,\chi)$ as in \eqref{10}. The second line comes from the contributions of $\tilde{\mu}_v$ with $v|Np$. It is nonzero unless $\det\tau_2$ is coprime to $p$. Hence, after applying the operator $U(L^2)$ we see that only those $\mathfrak{T}=\left[\begin{array}{cc}
L^2\tau_1 & \tau_2\\
-\epsilon\tau_2^{\ast} & \tau_4
\end{array}\right]$ with $\det \tau_2$ coprime to $p$ can contribute to the Fourier expansions. In particular $\mathfrak{T}>0$ with $\det\mathfrak{T}$ coprime to $p$. Therefore, by the explicit formulas for $a(\tau,\mathrm{diag}[y,v],\kappa,\chi)$ in \eqref{11},\eqref{12} we obtain the first part of the proposition. 

For $\chi=1$, note that
\[
\begin{aligned}
&\Lambda_{Np}^{2n}(\kappa,\chi)\sum_{\substack{0\leq i\leq n\\r=0,1}}(-1)^rp^{\frac{(2i+r)(2i+r-1)}{2}-2n(2i+r)}\mathfrak{E}_{\kappa}(z,w,\kappa,i,r)\\
=&\sum_{\tau\in\Lambda_{2n}}a(\tau,\mathrm{diag}[y,v],\kappa,\chi)e(\lambda(hz))e(2N^{-1}\lambda(\tau_2))\\
\times&\sum_{\substack{0\leq i\leq n\\r=0,1}}(-1)^rp^{\frac{(2i+r)(2i+r-1)}{2}-2n(2i+r)}\sum_j\sum_{x\in M_n(\mathcal{O})\hat{\delta}_{rij}/M_n(\mathcal{O}_p)}e(2\lambda(\tau_2x)).
\end{aligned}
\]
The second line comes from the contributions of $\mu^{i,r}_p$ and equals $1$ if $\det\tau_2\in\Z_p^{\times}$ and $0$ if otherwise. Hence, only those $\mathfrak{T}\in\Lambda_{2n}^+$ with $\det\mathfrak{T}$ coprime to $p$ contributes to the Fourier expansion of $\mathfrak{g}(z,w)$ and the proposition follows from the explicit formulas \eqref{11},\eqref{12}.
\end{proof}

\section{Special $L$-values and the $p$-adic Measure}
\label{section 6}

Let $f\in S_{\kappa}(\Gamma_0(N))$ be an eigenform with $\kappa>2n$. We further make following two assumptions:\\
(a) all primes $l\nmid N$ splits in $\mathbb{B}$ and $p$ is a fixed odd prime with $p\nmid N$ (thus also splits in $\mathbb{B}$),\\
(b) $f$ is $p$-ordinary with nonzero $p$-stabilisation $f_0\in S_{\kappa}(\Gamma_0(Np))$ and $f_0|U(L)=\alpha(L)f_0$ for any integer $L$ with $L|p^{\infty}$.

In previous sections, we interpolate the special $L$-values of $f$ by inner product formula \eqref{54} relate to Eisenstein series. The properties of these special values then follows from the explicit formulas of the Fourier expansion obtained in Proposition \ref{5.1}. In this section, we summarize these properties and construct the $p$-adic measure interpolating these special values.

Denote the $p$-Satake parameters of $f_0$ as $\beta_{1,p},...,\beta_{2n,p}$. Let $2n+1\leq t\leq \kappa$ be an integer and $v=\kappa-t$. Since the Fourier coefficients of $\mathfrak{g}(z,w)$ is holomorphic we have the following analytic property.

\begin{prop}
\label{6.1}
Let $\chi$ be a primitive Dirichlet character of conductor $p^{\mathfrak{c}}$. Then
\[
E_p(t,\chi)L_{Np}(t,f,\chi)
\]
has meromorphic continuation to whole complex plane and is holomorphic at $t=2n+1,...,\kappa$. 
\end{prop}

The algebraicity of special $L$-values for Case I is already studied in our previous paper \cite{T}. In this paper we are restrict ourselves to the case that our algebraic group is given by a totally isotropic hermitian form so the associated symmetric space will be a `tube' domain. Under this restriction the algebraicity of modular forms can be described in terms of Fourier expansion while for the general case discussed in \cite[Section 5.2]{T} we are using the description of CM points or Fourier-Jacobi expansions. The advantage of using the description via Fourier expansion is that we can now describe the explicit action of $\sigma\in\mathrm{Aut}(\C)$ as in the following proposition.

\begin{prop}
\label{6.2}
For any $\sigma\in\mathrm{Aut}(\C)$ we have\\
(Case I)
\begin{equation}
\begin{aligned}
\left(\frac{\Omega^{\mathfrak{c}}_{\kappa,v}(t)}{G(\chi)^nA_t(2m)}\frac{E_p(t,\chi)L_{Np}(f,t,\chi)}{\langle f_0,f_0\rangle}\right)^{\sigma}=\frac{\Omega^{\mathfrak{c}}_{\kappa,v}(t)}{G(\chi^{\sigma})^nA_{t}(2n)}\frac{E_p(t,\chi)L_{Np}(f^{\sigma},t,\chi^{\sigma})}{\langle f_0^{\sigma},f_0^{c\sigma c}\rangle},
\end{aligned}
\end{equation}
(Case II)
\begin{equation}
\begin{aligned}
\left(\frac{\Omega^{\mathfrak{c}}_{\kappa,v}(t)}{G(\chi)^{n}A_{t}(2m)}\frac{E_p(t,\chi)L_{Np}(f,t,\chi)}{g(\chi)\langle f_0,f_0\rangle}\right)^{\sigma}=\frac{\Omega^{\mathfrak{c}}_{\kappa,v}(t)}{G(\chi^{\sigma})^{n}A_{t}(2n)}\frac{E_p(t,\chi)L_{Np}(f^{\sigma},t,\chi^{\sigma})}{g(\chi^{\sigma})\langle f_0^{\sigma},f_0^{c\sigma c}\rangle}.
\end{aligned}
\end{equation}
for $t=2n+1,...,\kappa$. Here $f^c$ is the complex conjugation of $f$. 
\end{prop}

This can be proved by the same method as in the appendix of \cite{BS} using the Fourier expansion of $\mathfrak{g}(z,w)$ in Proposition \ref{5.1} and algebraic properties of Fourier coefficients in Proposition \ref{3.1}. As remarked in \cite{BS}, since only the Fourier coefficients of maximal rank involved in the Fourier expansion of $\mathfrak{g}(z,w)$ we actually obtain a better result that allowing us to discuss the special values below the absolute convergence bound of Eisenstein series. More precisely, here we are assuming $\kappa>2n$ while in \cite[Theorem 6.3]{T} we need $\kappa>4n-1$ (note that $n$ there is $2n$ here).

We now start the construction of the $p$-adic measure. Let $\C_p=\widehat{\overline{\Q}}_p$ be the completion of $\overline{\Q}_p$, and fix an embedding $\overline{\Q}\to\C_p$. The $p$-adic absolute value $|\cdot|_p$ naturally extends to $\C_p$, denote
\[
\mathcal{O}_{\C_p}=\{x\in\C_p:|x|_p\leq 1\}.
\]
A $p$-adic distribution on $\Z_p^{\times}$ is a continuous linear functional $\mu:C^{\infty}(\Z_p^{\times},\C_p)\to\C_p$. To construct a $p$-adic distribution it is enough to define $\mu$ on all sets of form $a+L\Z_p$ satisfying
\[
\mu(a+L\Z_p)=\sum_{b=0}^{p-1}\mu(a+bL+pL\Z_p).
\]
For such $\mu$ and a  locally constant function $f$, we denote $\mu(f)=\int_{\Z_p^{\times}}fd\mu$. A $p$-adic distribution is called a $p$-adic measure if it is bounded.

We will follow the strategy of constructing $p$-adic measures in \cite{CP}. That is we first define a $p$-adic distribution interpolate the special $L$-values. Then check the distribution satisfying the Kummer congruences so the $p$-adic measure exist. Finally, since the $p$-adic measure is uniquely determined by its value at all primitive Dirichlet character of $p$-power conductor, we conclude that the $p$-adic distribution we constructed is the desired $p$-adic measure. We recall the criterion of Kummer congruences as follow.

\begin{prop}
\label{6.3}
Let $\{f_i\}\subset C(\Z_p^{\times},\mathcal{O}_{\C_p})$ be a dense system of continuous functions, $\{a_i\}\subset\mathcal{O}_{\C_p}$ a system of elements. Then the existence of p-adic measure $\mu:C(\Z_p^{\times},\mathcal{O}_{\C_p})\to\mathcal{O}_{\C_p}$ such that
\[
\int_{\Z_p^{\times}}f_id\mu=a_i
\]
is equivalent to the following Kummer congruences: for an arbitrary choice of elements $b_i\in\C_p$ almost all of which vanish
\[
\sum_ib_if_i(y)\in p^m\mathcal{O}_{\C_p}\text{ for all }y\in\Z_p^{\times}\Longrightarrow\sum_ib_ia_i\in p^m\mathcal{O}_{\C_p}.
\]
\end{prop}

Recall that $\epsilon=-1$ in Case I and $\epsilon=1$ in Case II. For a function $\mathfrak{g}(z,w)$ on $M_k(\Gamma^0(N^2p),\psi)\otimes M_k(\Gamma^0(N^2p))$ let
\[
\mathcal{F}(\mathfrak{g})=\frac{\left\langle
\left\langle
\mathfrak{g}(z,-\overline{w}),f_0\left|_k\left[\begin{array}{cc}
0 & \epsilon\\
1 & 0
\end{array}\right]\right.
\right\rangle^w_{\Gamma^0(N^2p)},f_0\left|_k\left[\begin{array}{cc}
1 & 0\\
0 & N^2p
\end{array}\right]\right.
\right\rangle^z_{\Gamma^0(N^2p)}}{\langle f_0,f_0\rangle^2_{\Gamma_0(N^2p)}}.
\]
In particular, taking $\mathfrak{g}$ as in \eqref{51},\eqref{52} we obtain the desired special $L$-values by \eqref{54}. 

We first discuss Case I. Define
\[
\mathcal{H}^t_{a,L}(z,w)=A_{t}(2n)^{-1}\sum_{\chi}\chi(a)c_{\chi}^{-n(n-1)}G(\chi)^{-n}\cdot \mathfrak{g}(z,w,\kappa,v,\chi,L),
\]
where $\chi$ runs over all characters of conductor $c_{\chi}|L$. It is clear that from the Fourier expansion of $\mathfrak{g}$ that $\alpha(L)^{-4}\mathcal{F}(\mathcal{H}^t_{a,L})$ is independent of $L$. Thus if we set 
\[
\mu_t(a+L\Z_p):=\frac{p}{(p-1)L}\alpha(L)^{-4}\mathcal{F}(\mathcal{H}^t_{a,L}),
\]
then
\[
\sum_{b=0}^{p-1}\mu_t(a+bp^n+pL\Z_p)=\sum_{b=0}^{p-1}\frac{p}{(p-1)pL}\alpha(pL)^{-4}\mathcal{F}(\mathcal{H}_{a+bp^n,pL}^t)=\mu_t(a+L\Z_p).
\]
Therefore, $\mu_t$ is a $p$-adic distribution on $\Z_p^{\times}$ such that for primitive Dirichlet character $\chi$ of $p$-power conductor $c_{\chi}=p^{\mathfrak{c}}$,
\[
\begin{aligned}
\int_{\Z_p^{\times}}\chi d\mu_t&=\sum_{a\in(\Z/L\Z)^{\times}}\chi(a)\mu_t(a+L\Z_p)\\
&=A_t(2n)^{-1}c_{\chi}^{-n(n-1)}G(\chi)^{-n}\alpha(L)^{-4}\cdot\mathcal{F}(\mathfrak{g}(z,w,\kappa,v,\chi))\\
&=c_{\chi}^{-n(n-1)}G(\chi)^{-n}\alpha(p)\alpha(c_{\chi})^{-2}\frac{\Omega^{\mathfrak{c}}_{k,v}(t)}{A_t(2n)}\frac{E_p(t,\chi)L_{Np}(f_0,t,\chi)}{\langle f_0,f_0\rangle}.
\end{aligned}
\]
These distribution indeed take algebraic values and we view them as elements in $\C_p$. By the Fourier expansion of $\mathfrak{g}(z,w)$ in Proposition \ref{5.1}, we have
\begin{equation}
\mathcal{H}_{a,L}^t(z,w)=\sum_{\tau_1,\tau_4\in\Lambda_n^+}(N^2p)^{-2\kappa n}\alpha_{a,L}(\tau_1,\tau_4)e(2N^{-1}\lambda(\tau_2))e(\lambda((N^2p)^{-1}(\tau_1z+\tau_4w))),\label{59}
\end{equation}
with
\begin{equation}
\alpha_{a,L}(\tau_1,\tau_4)=\sum_{\tau_2}\mathfrak{P}_t^{v}(\mathfrak{T})\sum_{\chi}\chi(a\det(2\tau_2)^{-1})\prod_{l\nmid Np}P(\chi(l)l^{-t}).\label{60}
\end{equation}
Here $\mathfrak{T}=\left[\begin{array}{cc}
L^2\tau_1 & \tau_2\\
\tau_2^{\ast} & L^2\tau_4
\end{array}\right]$ and $p\nmid\det(\mathfrak{T})$ so the coefficient $\alpha_{a,L}(\tau_1,\tau_4)$ is actually $p$-integral. Put 
\begin{equation}
\label{kt}
\tilde{\mathcal{H}}_{a,L}^t(z,w)=a^{t-2n}\kappa_t\mathcal{H}_{a,L}^t(z,w)\text{ with }\kappa_t=\frac{c_{2n+1}^{\kappa-2n-1}}{c_t^{v}},
\end{equation}
and 
\begin{equation}
\tilde{\mu}_t(a+L\Z_p)=\frac{p}{(p-1)L}\alpha(L)^{-4}\mathcal{F}(\tilde{H}_{a,L}^t).
\end{equation}
Here the constant $c_s^v$ is defined in \eqref{csv}. Then $\tilde{\mathcal{H}}_{a,L}^t(z,w)\equiv 
a^{t-2n-1}\cdot \tilde{\mathcal{H}}_{a,L}^{2n+1}(z,w)\text{ mod }L$ and $d\tilde{\mu}_t(x)=x^{2n+1-t}d\tilde{\mu}_{2n+1}(x)$ by congruence \eqref{41}. Set
\[
\mu(a+L\Z_p)=\frac{p}{(p-1)L}\alpha(L)^{-4}\mathcal{F}(\tilde{\mathcal{H}}_{a,L}^{2n+1}).
\]

\begin{thm}
\label{6.4}
There is a unique $p$-adic measure $\mu$ on $\Z_p^{\times}$ such that for primitive Dirichlet character $\chi$ of $p$-power conductor $c_{\chi}=p^{\mathfrak{c}}$ we have
\begin{equation}
\int_{\Z_p^{\times}}\chi(x)x^{-t+2n}d\mu=\frac{c_{\chi}^{2nt-n(n-1)}}{\alpha(c_{\chi}^2)G(\chi)^n}\frac{\Gamma_{2n}(t)\Lambda_{\infty}(t)}{\pi^{2nt-n(2n-1)}}\frac{E_p(t,\chi)L_{Np}(f,t,\chi)}{\langle f_0,f_0\rangle}.
\end{equation}
for $t=2n+1,...,\kappa$. Here
\begin{equation}
\label{lambda1}
\Lambda_{\infty}(t)=\prod_{i=0}^{\kappa-t-1}C_{2n}\left(t+i-3n+\frac{1}{2}\right)C_{2n}(-t-i)\prod_{i=0}^{n-1}\frac{\Gamma(t+\kappa-2i)}{\Gamma(t+\kappa+n-2i)}.
\end{equation}
\end{thm}

\begin{proof}
By our construction, there is a distribution such that
\begin{equation}
\label{preinterpolate1}
\int_{\Z_p^{\times}}\chi(x)x^{-t+2n}d\mu=\frac{\alpha(p)}{\alpha(c_{\chi}^2)}\frac{\kappa_{t}}{c_{\chi}^{n(n-1)}G(\chi)^n}\frac{\Omega^{\mathfrak{c}}_{\kappa,v}(t)}{A_t(2n)}\frac{E_p(t,\chi)L_{Np}(f,t,\chi)}{\langle f_0,f_0\rangle}.
\end{equation}
We now prove it is indeed a $p$-adic measure. Since the Fourier coefficients of $\tilde{H}_{a,L}^{2n+1}$ is $p$-integral, by the same argument as \cite[Lemma 9.7]{BS}, $\mathcal{F}(\tilde{H}_{a,L}^{2n+1})$ takes values in $\mathcal{O}_{\C_p}$ up to a constant $C$. That is $C^{-1}\cdot\mu$ defines a measure $C(\Z_p^{\times},\mathcal{O}_{\C_p})\to\mathcal{O}_{\C_p}$. It suffices to check the Kummer congruences for $C^{-1}\cdot\mu$ or equivalently for $\mu_t$. Let $\{b_i\}\subset\C_p$ be a system of elements with almost of all vanish and assume
\[
\sum_ib_i\chi_i(x)\in p^m\mathcal{O}_{\C_p}\text{ for all }x\in\Z_p^{\times}.
\]
We need to show that
\[
\sum_ib_i\int_{\Z_p^{\times}}\chi_i(x)d\mu_t\in p^m\mathcal{O}_{\C_p}.
\]
This equals
\[
\mathcal{F}\left(A_t(2n)^{-1}\alpha(L)^{-4}\sum_ib_ic_{\chi_i}^{-n(n-1)}G(\chi_i)^{-n}\cdot \mathfrak{g}(z,w)\right).
\]
Taking the Fourier expansion of function under the bracket as \eqref{59}, its Fourier coefficients equals
\[
\begin{aligned}
\alpha_{a,L}(\tau_1,\tau_4)&=\alpha(L)^{-4}\sum_{\tau_2}\mathfrak{P}_t^{v}(\mathfrak{T})\det(\mathfrak{T})^{t-\frac{4n-1}{2}}\\
&\times \sum_{i}b_i\chi_i(a\det(2\tau_2)^{-1})\cdot\prod_{l\nmid Np}P(\chi_i(l)l^{-t}).
\end{aligned}
\]
The theorem follows since the first line in above formula is $p$-integral as $\alpha(L)$ is a $p$-adic unit by the ordinarity assumption. Recall that $\kappa_t, \Omega^{\mathfrak{c}}_{k,v}(t), A_t(2n)$ are constants defined in \eqref{kt}, \eqref{omega}, \eqref{A(2n)}. The straighforward computations provides us the interpolation formula as in the theorem from \eqref{preinterpolate1}.
\end{proof}

We now extend above discussion to Case II. As in symplectic case we need assumption on the parity of characters. In Case I the parity assumption is not necessary due to the fact that $\mathrm{det}(g)\in\R_+$ for $g\in\mathrm{GL}_n(\mathbb{H})$ (see \cite[Remark 3.1]{Q}).

\begin{thm}
\label{6.5}
There is a unique $p$-adic measure $\mu$ on $\Z_p^{\times}$ such that for primitive Dirichlet character $\chi$ of $p$-power conductor $c_{\chi}=p^{\mathfrak{c}}$ and $\chi(-1)=(-1)^{t}$ we have
\begin{equation}
\begin{aligned}
\int_{\Z_p^{\times}}\chi(x)x^{-t+2n}d\mu&=\frac{c_{\chi}^{(2n-1)t-n(n+1)}}{\alpha(c_{\chi}^2)}\frac{g(\chi)}{G(\chi)^n}\frac{\Gamma(t-2n)\Gamma_{4n}(t)\Lambda_{\infty}(t)}{\pi^{(4n+1)t-n(2n-1)}}\\
&\times\frac{1-\overline{\chi(p)}p^{t-2n-1}}{1-\chi(p)p^{2n-t}}\frac{E_p(t,\chi)L_{Np}(f_0,t,\chi)}{\langle f_0,f_0\rangle},
\end{aligned}
\end{equation}
Here $g(\chi)$ is the Gauss sum for Dirichlet character $\chi$ and
\begin{equation}
\label{lambda2}
\Lambda_{\infty}(t)=\prod_{i=0}^{\kappa-t-1}C_{2n}\left(t+i-2n+\frac{1}{2}\right)C_{2n}(-t-i)\prod_{i=0}^{2n-1}\frac{\Gamma\left(\frac{t+\kappa-i}{2}\right)}{\Gamma\left(\frac{t+\kappa+2n-i}{2}\right)}.
\end{equation}
\end{thm}

\begin{proof}
The proof is similar to Case I but note that the Fourier coefficient of $\mathfrak{g}(z,w)$ has a term of Dirichlet $L$-functions. For $2n+1\leq t\leq \kappa$, define
\[
\begin{aligned}
\mathcal{H}_{a,L}^t(z,w)&=A_t(2n)^{-1}\sum_{\chi}\chi(a)c_{\chi}^{-n(n-1)}G(\chi)^{-n}\\
&\times 2(2\pi i)^{2n-t}c_{\chi}^{t-2n}g(\chi)\Gamma(t-2n)\frac{1-\overline{\chi(p)}p^{t-2n-1}}{1-\chi(p)p^{2n-t}}\cdot\mathfrak{g}(z,w),
\end{aligned}
\]
and 
\[
\mu_t(a+L\Z_p)=\frac{p}{(p-1)L}\alpha(L)^{-4}\mathcal{F}(\mathcal{H}_{a,L}^t).
\]
Then $\mu_t$ is a $p$-adic distribution on $\Z_p^{\times}$ such that 
\[
\begin{aligned}
\int_{\Z_p^{\times}}\chi d\mu_t&=2(2\pi i)^{2n-t}\alpha(p)\alpha(c_{\chi})^{-2}c_{\chi}^{-n(n-1)}G(\chi)^{-n} \\
&\times c_{\chi}^{t-2n}g(\chi)\Gamma(t-2n)\frac{1-\overline{\chi(p)}p^{t-2n-1}}{1-\chi(p)p^{2n-t}}\frac{\Omega_{\kappa,v}(t)}{A_t(2n)}\frac{E_p(t,\chi)L_{Np}(f_0,t,\chi)}{\langle f_0,f_0\rangle}.
\end{aligned}
\]
Take the Fourier expansion of $\mathcal{H}_{a,L}^t(z,w)$ as in Case I. Its Fourier coefficient can be written as
\[
\begin{aligned}
\alpha_{a,L}(\tau_1,\tau_4)&=N^{-4\kappa n}\sum_{\tau_2}\det(\mathfrak{T})^{t-\frac{4n+1}{2}}\mathfrak{P}_t^{v}(\mathfrak{T})\sum_{\chi}\chi(a\det(2\tau_2))^{-1}\prod_{l\nmid Np}P(\chi(l)l^{-t})\\
&\times2(2\pi i)^{2n-t}c_{\chi}^{t-2n}g(\chi)\Gamma(t-2n)\left(1-\overline{\chi(p)}p^{t-2n-1}\right)L_{N}(t-2n,\chi\rho_{\mathfrak{T}}).
\end{aligned}
\]

We now check that the distribution $\mu_l$ satisfies Kummer congruences. Let $\{b_i\}\subset\C_p$ be a system of elements with almost of all vanish and assume
\[
\sum_ib_i\chi_i(x)\in p^m\mathcal{O}_{\C_p}\text{ for all }x\in\Z_p^{\times}.
\]
We need to show that
\[
\sum_ib_i\int_{\Z_p^{\times}}\chi_i(x)d\mu_l\in p^m\mathcal{O}_{\C_p}.
\]
This equals
\[
\mathcal{F}\left(2(2\pi i)^{t-2n}\Gamma(t-2n)A_t(2n)^{-l}\alpha(L)^{-4}\sum_ib_i\chi_i(a)c_{\chi_i}^{-n}G(\chi_i)^{-n}\right.
\]
\[
\left.c_{\chi_i}^{2n-t}g(\chi_i)\frac{1-\overline{\chi_i(p)}p^{t-2n-1}}{1-\chi_i(p)p^{2n-t}}\cdot\mathfrak{g}(z,w)\right).
\]
Taking the Fourier expansion of function under the bracket, its Fourier coefficients equals
\[
\begin{aligned}
\alpha_{a,L}(\tau_1,\tau_4)&=\alpha(L)^{-4}\sum_{\tau_2}\mathfrak{P}_t^{v}(\mathfrak{T})\det(\mathfrak{T})^{t-\frac{4n+1}{2}}\\
&\times\sum_ib_i(1-\overline{\chi_i(p)}p^{t-2n-1})L(1-t+2n,\overline{\chi_i\rho_{\mathfrak{T}}})\chi_i(a\det(4\tau_2))^{-1}\\
&\times\prod_{l\nmid Np}P(\chi_i(l)l^{-t})\cdot\prod_{l|N}\left(1-\chi_i\rho_{\mathfrak{T}}(l)l^{2n-t}\right).
\end{aligned}
\]
Here we are using the functional equation of Dirichlet $L$-functions. The sum over $i$ is in $p^m\mathcal{O}_{\C_p}$ by the Kummer congruences for Dirichlet $L$-functions (see \cite[Section 8]{BS} for $p$-adic measure attached to Dirichlet $L$-function). Gluing all these $\mu_t$ together as in Case I by using the congruence properties of $\mathfrak{P}_t^{v}(\mathfrak{T})$, we then obtain the desired $p$-adic measure $\mu$.  Recall the definition of  $\kappa_t, \Omega^{\mathfrak{c}}_{k,v}(t), A_t(2n)$ from \eqref{kt}, \eqref{omega}, \eqref{A(2n)} and The straighforward computations provides us the interpolation formula as in the theorem.
\end{proof}

\appendix
\section{A Transformation Property of Differential Operators}
\label{appendix}
In the appendix, we prove a transformation property of differential operators for hermitian matrices in the quaternionic case which will be used in Lemma \ref{lemma4.4}. The proofs here are almost same as the computations in \cite{Bo}.

We keep our notations as in Section \ref{section 4}. Recall that the space of quaternionic hermitian matrices
\[
S_n(\R)=\{y\in M_n(\mathbb{H}):y^{\ast}=y\}
\]
is identified with a subspace of complex hermitian matrices
\[
S_{2n}'(\R)=\{y\in M_{2n}(\C):y^{\ast}=y,J_n'y=\overline{y}J_n'\}.
\]
Denote $\partial_{ij}=\frac{\partial}{\partial y_{ij}}$ for basic differential operators and write them in a matrix $\partial=(\partial_{ij})$. For a (complex) hermitian matrix $\tau\in M_{2n}(\C)$ and $0\leq p\leq 2n$ we have
\[
\partial^{[p]}e^{\mathrm{tr}(\tau y)}=(\transpose{\tau}-J_n'\tau J_n')^{[p]}e^{\mathrm{tr}(\tau y)},\qquad y\in S_{2n}'(\R).
\]
Our computations will based on this basic identity and this cause the difference with \cite{Bo}.

\begin{lem}
Assume $y\in S'_{2n}(\R)$ is positive definite. The for $0\leq q\leq 2n$ and $s\in\C$ we have
\[
\partial^{[q]}\det(y)^s=C_q(s)\det(y)^s\cdot\transpose{y}^{-[q]},
\]
where 
\[
C_q(s):=s\left(s+\frac{1}{2}\right)...\left(s+\frac{q-1}{2}\right).
\]
\end{lem}

\begin{proof}
The proof is same as \cite[Proposition 3.5]{Bo}. It suffices to prove for $s=-2n$. Starting with \cite[Lemma 3.4]{Bo},
\[
\det(y)^{-2n}=\int_{M_{2n}(\C)}e^{-\pi\mathrm{tr}(x^{\ast}yx)}dx,
\]
and applying \cite[Lemma 3.3]{Bo} we have
\[
\begin{aligned}
\partial^{[q]}\det(y)^{-2n}&=\int_{M_{2n}(\C)}\partial^{[q]}e^{-\pi\mathrm{tr}(x^{\ast}yx)}dx\\
&=(-\pi)^q\int_{M_{2n}(\C)}(\transpose{(xx^{\ast})}-J_n'xx^{\ast}J_n')^{[q]}e^{-\pi\mathrm{tr}(x^{\ast}yx)}dx\\
&=(-\pi)^q\sum_{a+b=q}\left(\begin{array}{c}q\\a\end{array}\right)\frac{a!}{\pi^a}\left(\begin{array}{c}
2n\\a
\end{array}\right)\frac{b!}{\pi^b}\left(\begin{array}{c}
2n\\b
\end{array}\right)\det(y)^{-2n}(\transpose{y}-J_n'yJ_n)^{[-q]}\\
&=(-1)^qq!\left(\begin{array}{c}
4n\\
q
\end{array}\right)\det(y)^{-2n}\cdot2^{-q}\transpose y^{-[q]}\\
&=(-2n)(-2n+\frac{1}{2})...(-2n+\frac{q-1}{2}).
\end{aligned}
\]
as desired.
\end{proof}

We now prove the transformation property under $y\mapsto y^{-1}$.

\begin{prop}
For $0\leq q\leq 2n$, we consider the operator $D:=\transpose{y}^{[q]}\partial^{[q]}$ and the operator $\widehat{D}$ defined by $\widehat{D}(f):=D(\widehat{f})(Y^{-1})$ with $\widehat{f}(y):=f(y^{-1})$. Then
\[
\widehat{D}=(-1)^q\det(y)^{\frac{q-1}{2}}D^{\ast}\det(y)^{-\frac{q-1}{2}}
\]
with $D^{\ast}=\transpose(y^{[q]}\cdot\transpose{\partial}^{[q]})$.
\end{prop}

\begin{proof}
The proof is same as \cite[Theorem 3.6]{Bo}. Denote
\[
L:=\widehat{D}(f)(y),\qquad R:=(-1)^q\det(y)^{\frac{q-1}{2}}y^{[q]}\cdot\transpose{\partial}^{[q]}\det(y)^{-\frac{q-1}{2}}f.
\]
It suffices to prove $L=\transpose{R}$ for $f(y):=\det(y)^{2n}e^{\pi\mathrm{tr}(\tau y)}$ with $\tau$ (complex) hermitian and positive definite.

We calculate that
\[
R=(-1)^qf(y)y^{[q]}\sum_{a+b=q}\pi^b\left(\begin{array}{c}q\\a\end{array}\right)C_q(2n-\frac{q-1}{2})y^{-[a]}\sqcap(\transpose{\tau}-J_n'\tau J_n')^{[b]}.
\]
The same computations in \cite[Theorem 3.6]{Bo} for $L$ givens
\[
\partial^{[q]}(\widehat{f})(y^{-1})=(-\pi)^q2^{-q}f(y)C(q)\transpose{y}^{[q]}\cdot\transpose{y}^{-[a]}\sqcap(\transpose{\tau}-J_n'\tau J_n')^{[b]}
\]
where the factor
\[
\begin{aligned}
C(q)&=2^{-q}\sum_{q_1+q_2=q}\left(\begin{array}{c}
q\\
q_1
\end{array}\right)\sum_{\substack{a_1+b_1=q_1\\a_2+b_2=q_2}}\frac{a_1!}{\pi^{a_1}}\left(\begin{array}{c}
q_1\\
a_1
\end{array}\right)\left(\begin{array}{c}
2n-b_1\\
a_1
\end{array}\right)\frac{a_2!}{\pi^{a_2}}\left(\begin{array}{c}
q_1\\
a_2
\end{array}\right)\left(\begin{array}{c}
2n-b_2\\
a_2
\end{array}\right)\\
&=\sum_{a+b=q}\pi^{-a}\left(\begin{array}{c}
q\\
a
\end{array}\right)\sum_{\substack{a_1+a_2=a\\
b_1+b_2=b}}\left(\begin{array}{c}
a\\
a_1
\end{array}\right)\left(\begin{array}{c}
b\\
b_1
\end{array}\right)D_{a_1}(-2n+b_1)D_{a_2}(-2n+b_2)\\
&=\sum_{a+b=q}\pi^{-a}\left(\begin{array}{c}
q\\
a
\end{array}\right)2^{b}D_a(-4n+b-1).
\end{aligned}
\]
Here we are denoting $D_q(s):=s(s+1)...(s+q-1)$ as in \cite{Bo}. Comparing $L,R$ and note that
\[
(-\pi)^q2^{-q}\cdot\pi^{-a}2^bD_a(-4n+b)=(-1)^q\pi^bC_a(-2n+\frac{q-1}{2}),
\]
we obtain $L=\transpose{R}$ as desired.
\end{proof}

As a corollary we have (\cite[Corollary 3.7]{Bo}) 
\[
\partial^{[q]}(\widehat{f})=(-1)^q\transpose{y}^{-[q]}\det(y)^{-\frac{q-1}{2}}\left(\partial^{[q]}(f\cdot\det(y)^{-\frac{q-1}{2}})\right)(y^{-1})\cdot\transpose{y}^{-[q]}.
\]
Extending to the whole quaternionic upper half plane $\mathcal{H}_n$ (identified with $\mathfrak{H}_{2n}$) we obtain Lemma \ref{lemma4.4} as \cite[Proposition 3.8]{Bo}.

\section*{Acknowledgement}

This paper is prepared as part of my thesis at Durham University. I would like to express my gratitude to my supervisor, Thanasis Bouganis, for his helpful suggestions and discussions. I also thank the referee for providing several helpful comments.


\begin{thebibliography}{15}

\bibitem{B83}
S. Böcherer, Über die Fourier-Jacobi-Entwicklung Siegelscher Eisensteinreihen. Math. Z. 183 (1983), no. 1, 21–46.

\bibitem{B85}
S. Böcherer, Über die Fourier-Jacobi-Entwicklung Siegelscher Eisensteinreihen. II. Math. Z. 189 (1985), no. 1, 81–110.

\bibitem{Bo}
S. Böcherer, S. Das, On holomorphic differential operators equivariant for the inclusion of $\mathrm{Sp}(n,\R)$ in $\mathrm{U}(n,n)$. Int. Math. Res. Not. IMRN 2013, no. 11, 2534–2567.

\bibitem{BS}
S. Böcherer, C.-G. Schmidt, $p$-adic measures attached to Siegel modular forms. Ann. Inst. Fourier (Grenoble) 50 (2000), no. 5, 1375–1443.

\bibitem{BP}
S. Böcherer, A. Panchishkin, Higher twists and higher Gauss sums. Vietnam J. Math. 39 (2011), no. 3, 309–326.

\bibitem{Q}
T. Bouganis, On the standard $L$-function attached to quaternionic modular forms. J. Number Theory 222 (2021), 293–345.

\bibitem{T}
Bouganis T, Jin Y. Algebraicity of L-values attached to quaternionic modular forms. Canadian Journal of Mathematics. 2024;76(2):638-679. 


\bibitem{CP}
M. Courtieu, A. Panchishkin, Non-Archimedean $L$-functions and arithmetical Siegel modular forms. Second edition. Lecture Notes in Mathematics, 1471. Springer-Verlag, Berlin, 2004. viii+196 pp. ISBN: 3-540-40729-4.

\bibitem{F}
E. Freitag, Siegelsche Modulfunktionen.Grundlehren der mathematischen Wissenschaften 254. Springer-Verlag, Berlin, 1983. x+341 pp. ISBN: 3-540-11661-3.

\bibitem{G}
P. Garrett, Arithmetic automorphic forms for quaternion unitary groups, Thesis (Ph.D.)–Princeton
University. 1977. 61 pp.

\bibitem{Hida}
Hida, Haruzo, $p$-adic automorphic forms on Shimura varieties. Springer Monographs in Mathematics. Springer-Verlag, New York, 2004. xii+390 pp. ISBN: 0-387-20711-2.

\bibitem{Hua}
L.K. Hua, Harmonic analysis of functions of several complex variables in the classical domains, Translations of Mathematical Monographs, vol.6, Amer. Math. Soc. 1963.

\bibitem{K}
A. Krieg, Modular forms on half-spaces of quaternions, Lecture Notes in Mathematics, 1143. Springer-Verlag, Berlin, 1985. xiii+203 pp. ISBN: 3-540-15679-8.

\bibitem{LKW}
K-W. Lan, An example-based introduction to Shimura varieties, to appear in the proceedings of
the ETHZ Summer School on Motives and Complex Multiplication.

\bibitem{LZ}
Z. Liu, $p$-adic $L$-functions for ordinary families on symplectic groups. J. Inst. Math. Jussieu 19 (2020), no. 4, 1287–1347.

\bibitem{M}
J.S. Milne, Canonical models of mixed Shimura varieties and automorphic vector bundles, Automorphic Forms, Shimura Varieties and $L$-functions, Vol. I (Ann Arbor, MI, 1988), 283-414, Perspect, Math., 10, Academic Press, Boston, MA, 1990.

\bibitem{Sh82}
G. Shimura, Confluent hypergeometric functions on tube domains. Math. Ann. 260 (1982), no. 3, 269–302.

\bibitem{Sh94}
G. Shimura, Euler products and Fourier coefficients of automorphic forms on symplectic groups. Invent. Math. 116 (1994), no. 1-3, 531–576.

\bibitem{Sh95}
G. Shimura, Eisenstein series and zeta functions on symplectic groups, Invent. Math. 119 (1995), no. 3, 539-584.

\bibitem{Sh97}
G. Shimura, Euler Products and Eisenstein Series, CBMS Regional Conference Series in Mathematics, 93. American Mathematical Society, Providence, RI, 1997. xx+259 pp. ISBN: 0-8218-0574-6.

\bibitem{Sh99}
G. Shimura, Some exact formulas on quaternion unitary groups. J. Reine Angew. Math. 509 (1999), 67–102.

\bibitem{Sh00}
G. Shimura, Arithmeticity in the Theory of Automorphic Forms, Mathematical Surveys and Monographs, 82. American Mathematical Society, Providence, RI, 2000. x+302 pp. ISBN: 0-8218-2671-9.

\bibitem{SU}
C. Skinner, E. Urban,The Iwasawa main conjectures for $\mathrm{GL}_2$. Invent. Math. 195 (2014), no. 1, 1–277.

\bibitem{JV}
J. Voight, Quaternion algebras. Graduate Texts in Mathematics, 288. Springer, Cham, [2021],
©2021. xxiii+885 pp. ISBN: 978-3-030-56692-0; 978-3-030-56694-4.
\end{thebibliography}
\end{document}